\documentclass[11pt]{amsart}

\usepackage{graphicx}

\usepackage[margin=1in,marginparwidth=0.8in, marginparsep=0.1in]{geometry}

\usepackage{amsfonts,amssymb,latexsym,amsmath}
\usepackage{times}
\usepackage{mathrsfs}

\usepackage[backref=page, %
bookmarks=true, bookmarksopen=true,%
bookmarksdepth=3,bookmarksopenlevel=2,%
colorlinks=true,%
linkcolor=blue,%
citecolor=blue,%
filecolor=blue,%
menucolor=blue,%
urlcolor=blue]{hyperref}

\makeatletter
\def\@tocline#1#2#3#4#5#6#7{\relax
  \ifnum #1>\c@tocdepth 
  \else
    \par \addpenalty\@secpenalty\addvspace{#2}%
    \begingroup \hyphenpenalty\@M
    \@ifempty{#4}{%
      \@tempdima\csname r@tocindent\number#1\endcsname\relax
    }{%
      \@tempdima#4\relax
    }%
    \parindent\z@ \leftskip#3\relax \advance\leftskip\@tempdima\relax
    \rightskip\@pnumwidth plus4em \parfillskip-\@pnumwidth
    #5\leavevmode\hskip-\@tempdima
      \ifcase #1
       \or\or \hskip 1em \or \hskip 2em \else \hskip 3em \fi%
      #6\nobreak\relax
    \dotfill\hbox to\@pnumwidth{\@tocpagenum{#7}}\par
    \nobreak
    \endgroup
  \fi}
\makeatother

\usepackage{amsfonts}
\usepackage[english]{babel}
\usepackage[leqno]{amsmath}
\usepackage{amssymb,amsthm}
\usepackage{amscd}
\usepackage{enumerate}
\usepackage{epsfig}
\usepackage{mathabx}

\usepackage{color}

\definecolor{darkgreen}{rgb}{0.0, 0.7, 0.0}
\definecolor{purple}{rgb}{0.5, 0.0, 0.5}
\definecolor{red}{rgb}{0.8, 0.2, 0.0}

\usepackage[matrix,arrow,tips,curve]{xy}

\input{xy}
\xyoption{all}

\renewcommand{\AA}{\mathbb{A}}

\newcommand{\CC}{\mathbb{C}}
\newcommand{\DD}{\mathbb{D}}
\newcommand{\EE}{\mathbb{E}}
\newcommand{\FF}{\mathbb{F}}

\newcommand{\LL}{\mathbb{L}}

\newcommand{\PP}{\mathbb{P}}
\newcommand{\QQ}{\mathbb{Q}}

\newcommand{\UU}{\mathbb{U}}
\newcommand{\VV}{\mathbb{V}}
\newcommand{\WW}{\mathbb{W}}

\newcommand{\ZZ}{\mathbb{Z}}

\newcommand{\cC}{\mathcal{C}}
\newcommand{\fF}{\mathcal{F}}

\newcommand{\oO}{\mathcal{O}}

\newcommand{\im}{\mathrm{Im}\,}

\newtheorem{lemma}{Lemma}[section]
\newtheorem{proposition}[lemma]{Proposition}
\newtheorem{conjecture}[lemma]{Conjecture}
\newtheorem{corollary}[lemma]{Corollary}
\newtheorem{theorem}[lemma]{Theorem}
\newtheorem{fact}[lemma]{Fact}

\theoremstyle{definition}

\newtheorem{definition}[lemma]{Definition}

\newtheorem{setup}[lemma]{Setup}

\newtheorem{notation}[lemma]{Notation}

\theoremstyle{remark} 

\newtheorem{example}[lemma]{Example}
\newtheorem{remark}[lemma]{Remark}

\numberwithin{equation}{section}

\newenvironment{sis}{\left\{\begin{aligned}}{\end{aligned}\right.}

\DeclareMathOperator{\Ext}{Ext}

\DeclareMathOperator{\Def}{Def}
\DeclareMathOperator{\reg}{reg}

\theoremstyle{definition}

\theoremstyle{plain}

\newcommand{\Z}{\mathbb{Z}}
\newcommand{\Q}{\mathbb{Q}}

\newcommand{\un}{\underline}
\newcommand{\ov}{\overline}
\newcommand{\wt}{\widetilde}
\newcommand{\wh}{\widehat}

\newcommand{\Spf}{\rm Spf}

\def \Im{\rm Im}
\def \Supp{\rm Supp}

\def \id{\rm id}

\def \pol{\underline{m}}

\def \PP{\mathbb{P}}

\def \Gm{{\mathbb G}_m}

\def \C{\mathcal C}

\def\I{\mathcal I}

\renewcommand{\O}{\mathcal O}

\def\M0{\mathcal M^0}

\def \Tr{{\rm Tr}}

\def \M{\mathcal M}

\def \Breg{B_{\mathrm{reg} }}

\newcommand{\End}{{\operatorname{End}}}

\newcommand{\Var}{{\operatorname{Var}}}
\newcommand{\Rep}{{\operatorname{Rep}}}

\DeclareMathOperator{\codim}{{codim}}
\DeclareMathOperator{\Fr}{{Fr}}

\DeclareMathOperator{\Hom}{{Hom}}

\newcommand{\bbN}{{\mathbb N}}
\newcommand{\bbC}{{\mathbb C}}

\newcommand{\cF}{{\mathcal F}}
\newcommand{\cG}{{\mathcal G}}
\newcommand{\cH}{{\mathcal H}}

\newcommand{\cM}{{\mathcal M}}

\newcommand{\cO}{{\mathcal O}}
\newcommand{\cP}{{\mathcal P}}

\newcommand{\cS}{{\mathcal S}}

\newcommand{\rD}{\mathrm{D}}
\newcommand{\rC}{\mathrm{C}}
\newcommand{\rE}{\mathrm{E}}
\newcommand{\rV}{\mathrm{V}}

\renewcommand{\OE}{\overset{\rightarrow}{\rE}}

\newcommand{\el}[1][e]{\overset{\leftarrow}{#1}}
\newcommand{\er}[1][e]{\overset{\rightarrow}{#1}}

\newcommand{\Qlbar}{\overline{\QQ}_\ell}

\DeclareMathOperator{\Syst}{{Syst}}

\DeclareMathOperator{\Gal}{{Gal}}
\DeclareMathOperator{\Pic}{{Pic}}
\DeclareMathOperator{\Eff}{{Eff}}
\DeclareMathOperator{\Fit}{{Fit}}

\title{A support theorem for Hilbert schemes of planar curves, II}

\author{Luca Migliorini}
\address{Luca Migliorini, 
Dipartimento di Matematica, 
Universit\`a di Bologna,
Piazza di Porta S. Donato 5,
40126 Bologna,  ITALY. 
}
\email{luca.migliorini@unibo.it}
\author{Vivek Shende}
\address{Vivek Shende, 
Dept. of Mathematics,
University of California, Berkeley, 
970 Evans Hall,	
Berkeley CA 94720, USA.}
\email{vivek@math.berkeley.edu}

\author{Filippo Viviani}
\address{Filippo Viviani,
Dipartimento di Matematica e Fisica 
Universit\`a Roma Tre 
Largo San Leonardo Murialdo  
I-00146 Roma  Italy }
\email{viviani@mat.uniroma3.it}

\begin{document}

\begin{abstract}
We study the cohomology of Jacobians and Hilbert schemes 
of points on reduced and locally planar curves, which are
however allowed to be singular and reducible.  We show that the cohomologies of all Hilbert
schemes of all subcurves are encoded in the cohomologies of the fine compactified Jacobians of connected subcurves, via the perverse Leray filtration. 
\end{abstract}

\maketitle

\tableofcontents

\newpage

\section{Introduction}

Given an effective divisor $D$ on a nonsingular algebraic variety $\rC$, one can form the associated line bundle $ \oO_C(D)$,
thus defining a map from the space of effective divisors to the space of line bundles
\begin{eqnarray*} A: \Eff(\rC)=\coprod_{n\geq 0} \rC^{(n)}  & \to & \Pic(\rC) \\
D & \mapsto & \oO_C(D).
\end{eqnarray*}

For singular spaces, various changes must be made.  The spaces $\Eff(\rC)$
and $\Pic(\rC)$ still make sense, but the map does not.  Two problems
can already be seen when $\rC$ is a nodal curve:  the 
sheaf of functions with one pole at the node is not a line bundle, and the sheaf of functions
with double pole at the node has degree 3.

When $\rC$ is proper, reduced, and 
irreducible, there are natural substitutes \cite{DS, AIK, AK, AK2}.  
 The space of line bundles is extended to the space $\overline{\Pic}(\rC)$ of rank one, 
torsion free sheaves.  
The space of divisors is replaced by a space $\Syst(\rC)$ of generalized divisors -- 
rank one, torsion free sheaves equipped with injective sections.  There is an evident 
forgetful map $\Syst(\rC) \to \overline{\Pic}(\rC)$. 

When $\rC$ is proper of dimension 1 and locally planar, e.g. it lies on a smooth surface,these spaces behave in many ways like
their classical counterparts, $\overline{\Pic}(\rC)$ is
reduced and irreducible of dimension equal to the arithmetic genus of $\rC$, the space $\Syst(\rC)$
can be identified with the Hilbert scheme, and the above forgetful map is identified with the map 
sending a subscheme to the dual of its ideal sheaf 
\begin{eqnarray*} A: \coprod_{n\geq 0} \rC^{[n]} & \to & \overline{\Pic}(\rC) \\
D & \mapsto & \Hom_C(I_D, \oO_C).
\end{eqnarray*}

Reducibility introduces additional subtleties. 
Consider the curve consisting of two rational curves glued together at two points.  
The space of line bundles on this curve is $\ZZ \times \ZZ$ copies of $\mathbb{G}_m$, where the discrete
data gives the degrees of the line bundle on each component.  The ability to ``take the $(0,0)$ piece''
is lost in the compactification -- the torsion free sheaves coming from the nodes
serve to glue together the various components of degree $(a, d-a)$.  

The problem can be bounded by an appropriate choice of stability condition \cite{Gie, Ses, Sim}.
For locally planar curves, it is known that a generic choice leads to a fine moduli space, called a fine compactified Jacobian \cite{est1, MV, MRV1},
and moreover, that both its derived category \cite{MRV3} and the topological cohomology (see Theorem  \ref{thm:jac}) of the space 
do not depend on the choice of stability condition.  These naturally furnish invariants of the singular
curve; we will be interested here in investigating the latter. 

\vspace{2mm}

We begin with a nodal curve $\rC$. For simplicity in this introduction we assume all varieties are defined over the complex field. 
We write $\overline{J}_\rC$ 
for the fine compactified Jacobian determined by a fixed but unspecified generic stability condition.  
In the introduction, we restrict ourselves to the case where all components of $\rC$ are rational; 
for topological purposes, the general case differs from this only by the product of the Jacobians
of the components.  We write $\Gamma_\rC$ for the graph whose vertices are the
irreducible components of $\rC$ and whose edges are the nodes joining them. 

The space $\overline{J}_\rC$ 
is a union of toric varieties glued along toric divisors, by combinatorial rules which can
be given in terms of $\Gamma_\rC$ \cite{OS, A, MV}.  In particular,  
the zero dimensional torus orbits are in bijection with spanning trees of $\Gamma$.  
In terms of curves, a spanning tree is a connected  partial normalization of arithmetic
genus zero.  
That is: 
$$
\chi(\overline{J}_\rC) = \#\{\mbox{genus zero connected partial normalizations of a nodal curve $\rC$}\} $$
We will write this number as $n_0(\Gamma)$. 

A version of the above equality for irreducible curves was 
used by Yau, Zaslow, and Beauville to count curves on K3 surfaces \cite{YZ, Bea}.  
It has a certain physical meaning, further elaborated by Gopakumar and Vafa -- the right hand 
side has to do with topological string theory, and the left hand side has to do with BPS D-branes;
both are degenerations of some M-theoretic setup, so should be equal \cite{GV}.  They also 
explained that this reasoning explains how to generalize this formula to higher genus, by promoting the right hand side to the number $n_g(\Gamma)$ 
of genus $g$ connected spanning subgraphs of $\Gamma$, or equivalently,
the number of genus $g$ connected partial normalizations of the corresponding curve.


There are two ways to generalize the left hand side.  The first speaks only of the Jacobian, 
but introduces a filtration on its cohomology.  
Let $P^i H^*(\overline{J}_\rC, \QQ)$ be the local perverse Leray filtration, as defined in \cite{MS, MY}, on the cohomology
of the Jacobian, coming from spreading out over any versal deformation of $\rC$. 
Let $\LL=\QQ(-1)$ be the class of the affine line. 

\begin{theorem} \label{thm:nodaljac}
Let $C$ be a connected nodal curve over $\CC$ with rational components, and let $\Gamma$ be its   dual graph. 
Then we have the following equality in the Grothendieck group of Hodge structures: \begin{equation}
\sum_n q^n Gr_P^n H^*(\overline{J}_\rC, \QQ) = \sum_h n_h(\Gamma) \cdot (q\LL)^{g(\Gamma)-h} ((1-q)(1-q \LL))^{h} 
\end{equation}

\end{theorem}

In fact, the original Gopakumar-Vafa prediction spoke only of the specialization $\LL=1$; we are
giving a refined version.  This result follows from
Corollary \ref{cor:weightpolyic} combined with Theorem \ref{thm:jac} .
\vspace{2mm}

The second generalization of $\chi(\overline{J}_\rC)$ introduces new spaces instead of a
cohomological filtration.  In general, these spaces should be the $\Syst(\rC)$  above, or as 
Pandharipande and Thomas call them, $\mathrm{Pairs}(\rC)$ \cite{PT}.  When $\rC$ is Gorenstein, 
and in particular in the locally planar case to which we confine ourselves here,
these are isomorphic to the Hilbert schemes.  Unlike the Jacobians, the enumerative information
contained in these spaces is most naturally related to counting disconnected curves; the two are
conjecturally related by an exponential.  The pairs spaces were introduced to study
enumerative geometry on 3-folds \cite{PT, PT3}; but more relevant to our present work on locally 
planar curves are their uses in studying curves on surfaces \cite{S, KST, KT, KS, GS, GS2}, 
knot invariants \cite{ObS, ORS, GORS, DSV, DHS, M}, and the geometry
of the Hitchin system \cite{CDP}. 

We introduce some notation.  Form the group ring $\ZZ[[\ZZ^{vertices}]]$, i.e. 
the power series ring $\ZZ[[Q^{v_1}, Q^{v_2}, \ldots]]$ on the vertices of the graph.
This is where curve counting really happens, but as we count only reduced curves, 
we pass to the quotient by the ideal $(Q^{2v_1}, Q^{2v_2}, \ldots)$.  
On this quotient ring, we define an exponential 
\begin{eqnarray*}
\EE xp: (Q^{v_1}, Q^{v_2}, \ldots)/(Q^{2v_1}, Q^{2v_2}, \ldots)  & \to &
\ZZ[[Q^{v_1}, Q^{v_2}, \ldots]] / (Q^{2v_1}, Q^{2v_2}, \ldots) \\
\end{eqnarray*}
by sending $\EE xp(Q^v) = 1+ Q^v$, and requiring that sums go to products. 

For any subgraph $\Gamma' < \Gamma$, let $Q^{\Gamma'} := 
\prod_{v \in \Gamma'} Q^v$.  The Hilbert scheme
version of the formula is:

\begin{theorem} \label{thm:nodalhilb} 
Let $C$ be a connected nodal curve with rational components, with dual graph $\Gamma$. 
Then we have the following equality in the Grothendieck group of Hodge structures: 
\begin{equation}\label{eq:nodalhilb}
\sum_{\Gamma' < \Gamma} Q^{\Gamma'} (q\LL)^{1-g(\Gamma')}
\sum_{n=0}^\infty q^n H^*(C_{\Gamma'}^{[n]}, \Q)
= 
\EE xp \left(\sum_{\Gamma' < \Gamma} Q^{\Gamma'}
\sum_h n_h(\Gamma') \cdot \left(\frac{q\LL}{(1-q)(1-q\LL)} \right)^{1-h} 
\right) 
\end{equation}

Recall that, by definition, $n_h(\Gamma')$ vanishes when $\Gamma'$ is disconnected. 
\end{theorem}

Since the left hand side is computing cohomology of Hilbert schemes and the right hand side
is counting maps to the curve, this result is a sort of local motivic MNOP formula for 
maps to reduced nodal curves \cite{MNOP}.\footnote{The usual context of such formulas
is the counting of curves in 3-dimensional Calabi-Yau varieties, 
in which case the stable pairs moduli space is generally
singular and the Euler characteristics and cohomologies discussed here 
must be corrected by the Behrend
function \cite{B} or its cohomological upgrade.  The formulas here will apply to this 3-fold setting
only in the case that the moduli space is smooth, which however can happen, e.g. when the 3-fold
contains an isolated surface.} 

The result with $\Qlbar$ coefficients can be deduced by combining
Theorem \ref{thm:nodaljac} with Corollary \ref{Corlocal}. The result as stated follows by observing that
the mixed Hodge structures in  Equation \ref{eq:nodalhilb} are of Hodge-Tate type.


\begin{remark}
We do not know a formula for the Betti numbers of $\overline{J}_\rC$.
Finding such is nontrivial: while the space is built of toric varieties and carries
the action of a torus with finitely many fixed points, the cohomology is not equivariantly formal
-- in particular, there is cohomology in odd degrees.  
\end{remark}

\vspace{2mm}

We turn now to the more general setting of reduced planar curves.  Here, the $n_h(\rC)$ are more
mysterious.  The closest statement we know to a combinatorial interpretation operates
only at the level of Euler characteristics, and asserts that $\chi(n_h(\rC))$ is 
multiplicity of the loci of genus $h$ in a versal deformation of $\rC$ \cite{S}.  A conjectural
description of the refined invariants in terms of a real structure on the curve can be found
in \cite{GS}, where we also gave formulas in the case where
$\rC$ is a curve with an ADE singularity \cite{GS}.  From these it can be seen that 
$n_h(\rC)$ is a nontrivial Hodge structure, although
we know of no example in which it is not a polynomial in $\LL$.  

Nonetheless, we can at least ask for a relation between the analogues
of the left hand sides of Theorems \ref{thm:nodaljac} and \ref{thm:nodalhilb}.  

In the case of a single smooth curve $C$,  
the cohomologies of the Hilbert schemes $C^{[n]}$ -- in this case, just the symmetric products -- 
and the Jacobian $J(C)$ 
can both be built from $H^1(C, \QQ)$.  Explicitly:

$$\bigoplus_{n=0}^\infty q^n \cdot H^*(C^{[n]},  \QQ) =
\frac{\bigoplus q^i \cdot \bigwedge^i H^1(C,\QQ) [-i]}{(1-q)(1-q \LL)} = \frac{\bigoplus q^i H^i(J(C), \QQ) [-i]}{(1-q)(1-q \LL)}, $$
where $\LL := [-2](-1)$. 

The formula works in families: given a smooth family of curves $\pi_{sm}: \cC \to B_{sm}$, we have that

$$\bigoplus_{n=0}^\infty q^n \cdot R \pi_{sm*}^{[n]} \QQ = \frac{\bigoplus q^i \cdot \bigwedge^i R^1 \pi_{sm*} \QQ [-i]}{(1-q)(1-q \LL)} = \frac{\bigoplus q^i \cdot R^i \pi^J_{sm*} \QQ [-i]}{(1-q)(1-q \LL)}. $$

Now consider a family  $\pi_{\heartsuit}: \cC \to B_{\heartsuit}$ of reduced, {\em irreducible }
locally planar curves.   We can
form the relative Hilbert scheme $\pi^{[n]}_{\heartsuit}: \cC^{[n]} \to B_{\heartsuit}$, 
and the relative compactified 
Jacobian $\pi^J_{\heartsuit}: \overline{J}_\cC \to B_{\heartsuit}$.  
If all the relative Hilbert schemes 
have nonsingular total space, then the same is true for 
the relative compactified Jacobian.  
In \cite{MY, MS}, the families of cohomologies $R \pi^{[n]}_{\heartsuit*} \QQ$ and $R\pi^J_{\heartsuit*} \QQ$ were shown to enjoy the following relation:

$$\bigoplus_{n=0}^\infty q^n \cdot R \pi_{\heartsuit*}^{[n]} \QQ \cong
\frac{\bigoplus q^i \cdot IC(\bigwedge^i R^1 \pi_{sm*} \QQ) [-i]}{(1-q)(1-q \LL)} = \frac{\bigoplus q^i \cdot {}^p\!R^i \pi^J_{\heartsuit*} \QQ [-i]}{(1-q)(1-q \LL)}.$$

Here, $IC$ denotes the intersection cohomology sheaf extending the given local system and ${}^p\! R^i f_* := {}^p \mathcal{H}^i(Rf_*)$ means the $i$'th perverse
cohomology sheaf of the derived pushforward.   We take the convention that intersection cohomology complexes `begin in degree zero', so $K$ is perverse in our sense if $K[\dim B] $ is perverse in the sense of \cite{BBD}, see \S  \ref{sec:CKS_preliminaries} .

We recall a few ideas from the proof. 
It follows from the ``decomposition theorem'' of \cite{BBD} that the middle term above is a direct summand both on the right and the left, and any other summands
must have positive codimensional support, so it remains only
to show that there are no such summands.  On the RHS, hence on the LHS for $n \gg 0$ via the Abel-Jacobi map,
this is a consequence of the `support theorem' of \cite{N1}.  In \cite{MY}, this is bootstrapped to an argument
for the LHS by constructing correspondences between the Hilbert schemes.  In \cite{MS}, we take a different approach, suitable for both the LHS and RHS,
to reduce checking to the nodal locus, where it may be done explicitly.  We have since abstracted this method
into the theory of higher discriminants \cite{MS2}.  Yet another approach to similar results can be found in \cite{R}.

Our present goal is to establish such a comparison over the locus of reduced curves -- i.e., to treat the reduced but not necessarily irreducible case.  
As we already mentioned, there are already subtleties in the definition of the compactified Jacobian,
but so long as the curves lie in a fixed surface or fixed family of surfaces or we are working \'etale locally over the base, we can choose compatible stability conditions over the whole base and consider the
relative fine compactified Jacobian $\pi^J: \overline{J}_\C \to B$ (see Theorem \ref{T:fam-Jac}).
Second, due to the above stability issues, there is no Abel-Jacobi map directly relating the Hilbert
schemes and the Jacobians.
Third, it is no longer true in general that smoothness of $\overline{J}_\C$ guarantees the absence
of summands of $R\pi^J_* \QQ$ with positive codimensional supports. 

\begin{example}
Consider a one-parameter family of elliptic curves degenerating to a cycle of $\ge 2$ $\PP^1$'s.  
This family is its own relative fine compactified Jacobian \cite[Prop. 7.3]{MRV1},
but evidently $R\pi^J_* \QQ$ has
a summand supported at the special point to account for its extra $H^2$.  
\end{example}

Nonetheless, over sufficiently big families, this phenomenon does not occur.  

\begin{definition}\label{def:Hsmooth}
We say $\pi: \cC \to B$ is {\em H-smooth} if all relative Hilbert schemes
$\cC^{[n]}$ have smooth total space.  Note this includes $\cC^{[0]} = B$. 
\end{definition}

\begin{example}
Over any field, a versal family of reduced curves with locally planar singularities is H-smooth, see \S \ref{section:compjac_nvf}
for the general discussion of the condition of H-smoothness, based on the results in \cite{S}. 
\end{example}

\begin{theorem} \label{thm:jac}
Let  $\pi: \cC \to B$ be H-smooth.  Then no summand of
$R\pi^J_* \QQ$ has positive codimensional support. 
Thus, 
${}^p\!R^i \pi_{*}^J \QQ \cong IC(\bigwedge^i R^1 \pi_{ sm *} \QQ)$, and 
the stalk at $[\rC]$ of ${}^p R^i \pi_{*}^J \QQ$ does not depend
on the choice of the H-smooth family $\cC$.  
\end{theorem}


In some cases, this follows from the work of Chaudouard and Laumon \cite{CL}.
To prove the result, we use the method of higher discriminants \cite{MS2}, plus the following 
smoothness criterion, to reduce the result to the case of irreducible curves, where it is known \cite{MS}. 

\begin{theorem}\label{nonsing_rel_compjac_fd}
Let $\pi: (\cC, \rC) \to (S,b)$ be a projective flat family of connected locally planar curves, 
with distinguished special fibre $\rC=\cC_b$.  
Let $k^{\rm loc}_{\pi, b}: T_b S \to T \Def^{\rm loc}(\rC)$ be the induced map to the 
first-order deformation of the singularities of $\rC$.  
Let $\gamma(\rC)$ be the number of connected components of $\rC$, and $\delta(\rC)$ its cogenus. 

If $\,\Im(k^{\rm loc}_{\pi, b})$ is a generic subspace of $T \Def^{\rm loc}(\rC)$ of dimension at least $\delta(\rC) + 1 - \gamma(\rC)$, 
then the relative compactified Jacobian $\ov{J}_{\cC}$  is regular along the special fibre 
$\ov{J}_{\rC}$.
\end{theorem}
A more precise version of Theorem \ref{nonsing_rel_compjac_fd} can be found as Theorem \ref{nonsing_rel_compjac}.

On the other hand, even for versal families, there are many
summands of $R\pi^{[n]}_* \QQ$ which are supported in positive codimension.   
In fact, at a reducible curve $[\rC] \in B$,
there is such a summand for every splitting of $\rC$ into connected subcurves.
The simplest example is given by a one-dimensional family of nonsingular conics degenerating to a reducible one.
The family is versal, and already $R\pi^{[1]}_* \QQ$ has a summand supported at the central point.   
Nonetheless, we will establish various analogues of the main result of \cite{MY, MS}, both at a 
single curve, and globally for what we call {\em independently broken} H-smooth families, see \S \ref{ind_brok_H_smooth} for the definition. 

\medskip
We now describe these results, treating for simplicity only the case of a versal family of locally planar curves. Our results hold for cohomology with $\Qlbar$ coefficients since our methods of proof depends on reduction to positive characteristic.

Let   $\rC$ be a locally planar curve, let $V$ be the set of irreducible components 
and let $(\cC,\rC)\to (B,b)$ be a versal deformation  of $\rC$, small enough so that there  is no monodromy of the irreducible components of $\rC$ in the equigeneric stratum, see Lemma \ref{irred_comp}. 
By considering specialization to the central fibre, the base  $ B$ is stratified by the
closed subsets  $\overline{B_\lambda} \hookrightarrow B$, where $\lambda$  is a partition of $V$, 
corresponding to decompositions $\rC=\bigcup \rC_i$ into connected subcurves. For every $\lambda$ 
we consider the open dense subset $B_\lambda \subseteq \overline{B_\lambda}$ parameterizing nodal curves in  $\overline{B_\lambda}$.
Over $B_\lambda$ the nodes separating the different subcurves persist, and can therefore be normalized, thus giving a family of partial normalizations
$\pi_\lambda: \cC_{\lambda} \to B_{\lambda}.$  


We have the dense, open subsets $B_{\lambda, reg} \subseteq B_\lambda$ where the morphism 
$$
\cC_{\lambda, reg}:={\cC_{\lambda}}_{|{B_{\lambda, reg}}} \to B_{\lambda, reg}
$$
is smooth. Denote by $\iota_\lambda: B_{\lambda, reg} \to B$ the natural inclusions.

We consider the associated symmetric product families 
$$
\pi_\lambda^{[r]}: \cC_{\lambda, reg}^{[r]} \to B_{\lambda, reg},
$$
which are still smooth, so that 
$$R {\pi_\lambda^{[n]}}_* \Qlbar \simeq \bigoplus_i R^i{\pi_\lambda^{[n]}}_* \Qlbar[-i],$$
a direct sum of (pure, semisimple) shifted local systems on $B_{\lambda, reg}$.
Set 
\[
\fF_{\lambda}^{[n]}:=\bigoplus_i \left(\left({\iota_\lambda}\right)_{!*}R^i{\pi_\lambda^{[n]}}_* \Qlbar \right)[-i],
\]
a complex of sheaves supported on $\overline{B_\lambda}$.
 Then we have
\begin{theorem} \label{thm:hilb_intro}
$$R\pi_*^{[n]} \Qlbar \cong \bigoplus_{\lambda \in \cP}  \fF_{\lambda}^{[n-\delta(\lambda)]} [-2\delta(\lambda)](\delta(\lambda))$$
where $\cP$ is the set of partitions of $V$ decomposing $\rC$ in connected subcurves, and $\delta(\lambda)$ is the number of nodes being normalized in the stratum
$B_\lambda$.
\end{theorem}
In Example \ref{ex:vivekformula} this formula is made explicit for the versal deformation of a pair of incident lines.
The notion of {\em higher discriminants of a map} developed in \cite{MS2} and the fact that nodal curves are dense in 
these higher discriminants, which are determined via deformation theory relying on \cite{S}, reduce the proof of 
the theorem to nodal curves.
To identify the two sides of (\ref{thm:hilb}) for a versal deformation of a nodal curve $\rC$ we pass to a family defined 
over a finite field $\FF_{\pi},$ and compute, at every point in the base, the trace of the Frobenius map and its iterates 
on the stalk of the right hand side of the equality and we compare them with the counting of points in the fibres of $\pi^{[n]}$ over the extensions of $\FF_{\pi}$.
Then we conclude by the Grothendieck-Lefschetz formula and Chebotarev theorem (this is why we require
$\Qlbar$ coefficients). Determining
the traces for the  sheaves $IC(\bigwedge^i R^1 \pi_{ sm *}  \overline{\QQ}_l)$ is
the essential computation, which we perform in Section \ref{sec:CKS} using the Cattani-Kaplan-Schmidt complex \cite{cks}.


To relate this result with the discussion above, especially with formula \ref{thm:nodalhilb}, note that we have an `exponential map' which acts on the category of sheaves on $\coprod_{\lambda} B_{\lambda}$ by 
\begin{equation}\label{defexp}
\mathbb{E}xp(\cF)|_{B_{\lambda}} := \bigoplus_{\mu \succcurlyeq \lambda} \bigboxtimes (\cF|_{B_{\mu}}).
\end{equation}
With this notation, our main Theorem reads:

\begin{theorem}\label{thm:main_intro} 
Let $\cC \to B$ be a projective versal family of locally planar curves
admitting
relative fine compactified Jacobians $\overline{J}_\C \to B$ 
(the relative fine compactified Jacobian of a disconnected curve 
is set to be empty by definition)   
Let $g$ denote the locally constant function 
giving the arithmetic genus of the curves being parameterized. 

Then there are isomorphisms in $D^b_c(\coprod B_\lambda)[[q]]$:
\begin{eqnarray*}
(q \LL)^{1-g} \bigoplus_{n=0}^\infty q^n R\pi^{[n]}_{*} \overline{\QQ}_l & \cong &
\mathbb{E}xp \left( 
 (q\LL)^{1-g} \cdot \frac{  \bigoplus
 q^i \cdot IC(\bigwedge^i R^1 \pi_{ sm *}  \overline{\QQ}_l  )[-i]}{(1-q)(1-q \LL)} \right)
 \\
 & \cong &
 \mathbb{E}xp \left(  (q\LL)^{1-g} \cdot \frac{  
 \bigoplus
 q^i \cdot {}^p R^i \pi_{*}^J  \overline{\QQ}_l  [-i]}{(1-q)(1-q \LL)} \right).
\end{eqnarray*}
\end{theorem}

By taking the stalks, Theorem \ref{thm:main_intro} has the following local corollary,

\begin{corollary}\label{Corlocal}
Let $\rC$ be a reduced planar curve.  We write $\rC' < \rC$ to indicate a subcurve. There
is an isomorphism 
$$
\bigoplus_{\rC' < \rC} Q^{\rC'} (q\LL)^{1-g(\rC')}
\bigoplus_{n=0}^\infty q^n H^*((\rC')^{[n]}; \Qlbar) 
= 
\EE xp \left(\sum_{\rC' < \rC} \frac{Q^{\rC'} (q\LL)^{1-g(\rC')}}{(1-q)(1-q\LL)}
\bigoplus_i q^i Gr_P^i H^*(\overline{J}_{\rC'};\Qlbar) 
\right) 
$$
Here, $Gr_P^i H^*(\overline{J}_{\rC'};\Qlbar)$ is by definition
${}^p R^i \pi_{*}^J  \Qlbar [-i]|_{[\rC']}$ with respect to any H-smooth family containing $\rC'$ and $\overline{J}_{\rC'}$ is any fine compactified Jacobian of $C'$ 
(with the convention that $\overline{J}_{\rC'}$ is the empty set for disconnected $\rC'$). 
\end{corollary}


The point of these results is that the perverse filtration appears prominently in recent
studies of the cohomology of the Hitchin system \cite{dCHM, CDP} and its fibres \cite{GORS, OY}, 
but is difficult to compute directly.  On the other hand, the cohomology of the 
Hilbert schemes is more directly accessible, and the theorem explains how to recover the 
associated graded pieces of the perverse filtration on the Jacobian from the collection of
all cohomologies of the Hilbert schemes.  

This sort of relation was in a certain sense predicted
in the physics literature \cite{GV, KKV, HST, CDP} as a relation between refined Gopakumar-Vafa invariants
(here, the Jacobians) and the refined Donaldson-Thomas invariants (here, the Hilbert schemes).  

\vspace{2mm} 
\noindent {\bf Acknowledgements.} 
We thank Dan Abramovich, Riccardo Grandi, Tam\'as Hausel and Jochen Heinloth for helpful discussions. Special thanks to Mark A. de Cataldo, 
who, on many occasions, pointed out to the first named author many misconceptions about the 
decomposition theorem over a finite field, and helped to correct the mistakes arising from them.
L.M. is partially supported by PRIN project 2015 ``Spazi di moduli e teoria di Lie".
During the (long) preparation of this paper L.M. was a member of the School of Mathematics of 
the Institute for Advanced Study in Princeton, partially funded by the 
Giorgio and Elena Petronio  fellowship.
V. S. is supported by the NSF grant DMS-1406871, and by a Sloan fellowship. F.V. is partially supported by the MIUR project  ``Spazi di moduli e applicazioni'' (FIRB 2012).

\section{Background}

\subsection{Notation}

\subsubsection{} \label{N:curves}
A \textbf{curve}  is a \emph{reduced} (but not necessarily geometrically irreducible)
scheme  of pure dimension $1$ over a \emph{perfect} field $k$.  In practice we take
$k$ to be the complex numbers ($\CC$), a finite field ($\FF_{\pi}$), or the algebraic 
closure of a finite field ($\overline{\FF}_{\pi}$). 

Unless otherwise specified, a curve is meant to be projective.

\subsubsection{} 
A \textbf{family of curves}  $\pi:\C\to B$ is a flat and proper morphism of $k$-schemes all of whose geometric fibers are curves.  
If $\pi$ is a projective morphism,  we say that the family is projective.

\subsubsection{} 
Given a curve $\rC$, we denote by $\rC_{\rm sm}$ the smooth locus of 
$\rC$, by $\rC_{\rm sing}$ its singular locus, by $\nu:\rC^{\nu}\to \rC$ 
the normalization morphism, and by $V(\rC)=\pi_0(\rC_{\rm sm})=\pi_0(\rC^{\nu})$ 
the set of its irreducible components: $\rC=\bigcup_{v\in V(\rC)} \rC_v$. 

\subsubsection{}\label{invaria}

We employ the following names and notation for numerical invariants of a curve $\rC$:

\vspace{4mm}
\begin{tabular}{l|l|l}
 name & notation & formula \\
 \hline
 number of irreducible components & $\gamma(\rC)$ & \\
 arithmetic genus & $g(\rC)$ & $1-\chi(\O_\rC)$ \\
 geometric genus & $g(\rC^{\nu})$ & \\
 cogenus, or total delta invariant & $\delta(\rC)$ & $g(\rC)-g(\rC^{\nu})$ \\
 abelian rank & $g^\nu(\rC)$ &  $g({\rC^\nu}) - 1 + \gamma(C)$ \\
 affine rank &  $\delta^a(\rC)$ & $\delta(C)+1-\gamma(C) = g(\rC)-g^{\nu}(\rC)$
\end{tabular}
\vspace{4mm} 
 
Recall that the cogenus is equal to the 
sum of the local delta invariants of the singularities:
$$\displaystyle \delta(\rC):=\sum_{q\in \rC_{\rm sing}} [k(q):k]\, \cdot \, \delta(\rC,q)=
\sum_{q\in \rC_{\rm sing}}[k(q):k]\, \cdot \, {\rm length}(\nu_*\O_{\rC^{\nu}}/\O_{\rC})_q.$$

The terminology ``affine rank'' and ``abelian rank'' will be explained in \ref{N:Jac-gen}. 
Note the abelian rank is also equal to 
the sum of the genera of the connected components of the normalization. 

The cogenus $\delta(\rC)$ and the affine rank $\delta^a(\rC)$ are upper semicontinuous in families of curves (see \cite[Prop. 2.4]{DH} or \cite[Chap. II, Thm. 2.54]{GLS} in characteristic zero and \cite[Prop. A.2.1]{L-jc} and \cite[Lem. 3.2]{MRV2} in
arbitrary characteristic). 
 Equivalently, the geometric genus and the abelian rank are lower semicontinuous.

\subsubsection{}
A curve $\rC$ is \textbf{locally planar at $p\in \rC$} if  the completion
$\wh{\O}_{\rC,p}$ of the local ring of $\rC$ at $p$ has embedded dimension two, i.e., 
$\wh{\O}_{\rC,p}\cong k[[x,y]]/(f),$ for some reduced $f=f(x,y)\in k[[x,y]]$.

A curve $\rC$ is locally planar if it is locally planar at every $p\in \rC$.  Being locally a divisor
in a smooth space, a locally planar
curve is Gorenstein, i.e. the dualizing sheaf $\omega_\rC$ is a line bundle.

\subsubsection{} 
A \textbf{subcurve} $D$ of a curve $\rC$ is a reduced subscheme of pure dimension $1$.  We say that a sub-curve $D\subseteq \rC$ is non-trivial if $D\neq \emptyset, \rC$.

\subsubsection{} \label{N:Jac-gen}
Given a curve $\rC$, the \textbf{generalized Jacobian} of $\rC$, denoted by $J_\rC$ or by $\Pic^{\un 0}(\rC)$,
is the connected component of the Picard scheme $\Pic(\rC)$ of $\rC$ containing the identity, see \cite[\S 8.2, Thm. 3]{BLR} and references therein for existence theorems.
The generalized Jacobian of $\rC$ is a connected commutative smooth algebraic group of dimension equal to
$h^1(\rC,\O_\rC)$. Under mild hypotheses such as existence of a rational $k$-point, or triviality of the Brauer group of $k$, certainly met in the cases $k=\FF_\pi, \ov{\FF}_\pi,\CC$,
its group of $k'$-valued points, for $k'$ a finite extension of $k$, parameterizes line bundles on $\rC$, defined over  $k'$, of multidegree $\un 0$ (i.e. having
degree $0$ on each irreducible component of $\rC$)  with the multiplication given by the tensor product.

From the exact sequence of sheaves on $\rC$
$$1 \longrightarrow {\Gm} \longrightarrow \nu_* {\Gm} \longrightarrow \nu_* {\Gm} /  {\Gm} \longrightarrow 1$$ 
where $\nu:\rC^{\nu}\to \rC$ the normalization morphism, it follows easily that the generalized Jacobian $J_{\rC}$ is an extension of an abelian variety of dimension $g^{\nu}(\rC)$ (namely the Jacobian of the normalization $\rC^{\nu}$) 
by an affine algebraic group of dimension equal to $\delta^a(\rC)$.

\subsubsection{} 
We use $\LL$ to mean ``whatever incarnation of the Lefschetz motive is appropriate''.  That is,
if we are discussing ungraded vector spaces in the presence of weights, e.g. the K-group of mixed Hodge structures or of
continuous $\hat{\ZZ}$ representations over $\Qlbar$, we mean a one dimensional vector space twisted by $(-1)$.  If we are working
with graded vector spaces in the presence of weights, i.e. in the derived category of the above rather than the K-group, we mean a one dimensional vector space,
twisted by $(-1)$, and placed in cohomological degree $2$, e.g. $\LL=\Qlbar(-1)[-2]$.
In the Grothendieck ring of varieties $\LL$ is  the class of the
affine line.

\subsection{The Cattani-Kaplan-Schmid complex} \label{sec:CKS_preliminaries}
In this paper we use the convention according to which the {\bf intersection cohomology complex} $IC(L)$ of a local system $L$ on a dense open set $Z^0$ of a nonsingular variety $Z$ restricts to $L$, as opposed to 
$L[\dim Z]$.  In our convention we say $K$ is {\bf perverse} on $Z$  if and only if $K[\dim Z]$ is perverse in the sense of \cite{BBD}.
Thus, given a local system $L'$ on a locally closed  $Z' \subset Z$, the complex $IC(L')[- \mathrm{codim}Z']$ is perverse.

If $\mathscr L$ is a unipotent local system underlying a variation of pure Hodge structures of weight $w$ on a product of punctured polydisks $(\DD^*)^r \subset \DD^r$,
the paper \cite[\S 1]{cks}, gives a model for the stalk $IC(\mathscr L)_0$ at $0 \in \DD^r$ of the intersection cohomology complex of  $\mathscr L$ and its weight filtration (see also
\cite[\S 3]{saito2}).
This model works just as well in the $\ell$-adic \'etale theory,  and we shortly review it here, as it plays a central role in our computations.
 According to our conventions the intersection cohomology complex lives in degrees $ [0, \ldots, $ $\dim Y -1]$.
Assume $Y$ is a regular scheme over ${\FF}_\pi$, and $D=\bigcup_{j \in J} D_j$ is a normal crossing divisor. After  \'etale localization we may assume that $Y$ is some Zariski neighborhood of the origin in $\AA^n$, with coordinate functions $t_1, \ldots, t_n$, and $D$ is defined by the equation $\prod_{j \in J}t_j=0$, with $J=\{1, \ldots , k\}$.
We denote $j:Y\setminus D \to Y$.

Let $\mathscr L$ be  a ``lisse" {\em unipotent} sheaf on $Y \setminus D$, tamely ramified  along $D$, pointwise pure of weight $w$.

Let $\Psi_1, \ldots, \Psi_k$ be the nearby-cycle functors associated with the functions $t_1, \ldots , t_k$, and denote 
\begin{equation*}\label{psi_mixed}
\Psi= {\Psi}_{1} \circ \ldots \circ{\Psi}_k.
\end{equation*}
Thus 
${\mathrm L}:= \Psi ({\mathscr L})$ is a lisse mixed sheaf on $E:= \bigcap_{j \in J} D_j$, endowed with {\em commuting nilpotent} endomorphisms 
$N_j: {\mathrm L} \to {\mathrm L}(-1)$. The weights are given in terms of the monodromy filtration of a general element $\sum a_jN_j$, as explained in \cite{cks}.

\begin{proposition}\label{cks_prop}

We have the following isomorphism for the restriction of the intersection cohomology complex to $E$:
\begin{equation*}\label{cks}
i_E^*IC(\mathscr{L}) \simeq {\mathbf C}^{\bullet}( \{ N_j \}, \mathscr{L}):=\{
0 \to  {\mathrm L} \to \bigoplus_{|I|=1} \im N_I  \to \bigoplus_{|I|=2} \im N_I\to \cdots \to  \im N_J  \to 0 \}
\end{equation*}
where the differentials are given by
\begin{equation*}\label{diff_cks}
(-1)^{k}N_i:\im N_{i_1}\cdots N_{i_k}\to \im  N_iN_{i_1}\cdots N_{i_k} \hbox{ if } i \neq \{i_1, \cdots , i_k\}.
\end{equation*}
\end{proposition}

\subsection{Deformation theory of locally planar curves}\label{S:defo}

We recall facts about the deformation theory of locally planar curves 
and their simultaneous desingularization. 
These facts are well known over the complex numbers; original proofs 
can be found in the papers
 \cite{T, DH} and a textbook treatment in \cite{GLS}.  They have also been 
 partially extended  to positive characteristic  in \cite{L-jc}, \cite{MY}, \cite{MRV2}.
For maximal accessibility, we give  precise references
to the book of Sernesi \cite{Ser} for some of the standard deformation theoretic facts
we use.  

Let $\Def_\rC$ be the \emph{deformation functor} of a (reduced and projective) curve $\rC$ (\cite[Sec. 2.4.1]{Ser}). For $p\in \rC_{\rm sing}$, we denote by $\Def_{\rC,p}$ the
deformation functor of the complete local $k$-algebra $\wh{\O}_{\rC,p}$ (\cite[Sec. 1.2.2]{Ser}).
There is a natural transformation of functors
\begin{equation}\label{E:mor-func}
\Def_\rC\to \Def_\rC^{\rm loc}:=\prod_{p\in \rC_{\rm sing}} \Def_{\rC,p}.
\end{equation}
If $\rC$  has locally planar singularities (or, more generally, locally complete intersection singularities), the  functors $\Def_\rC$ and $\Def_{\rC}^{\rm loc}$ are smooth (\cite[Cor. 3.1.13(ii) and Ex. 2.4.9]{Ser})  and 
the morphism \eqref{E:mor-func} is  smooth (\cite[Prop. 2.3.6]{Ser}).

Given any deformation $\pi:(\C,\rC)\to (B, b)$ of $\rC$, i.e. a family of curves $\pi:\C\to B$  together with a $k$-point $b\in B$ such that $\rC=\C_b:=\pi^{-1}(b)$, by pulling back $\pi$ via the natural morphism 
$\Spf \hat{\O}_{B,b}\to B$ (where $\Spf$ to denote the formal spectrum), we get a formal deformation of $\rC$ over $\hat{\O}_{B,b}$, which induces a morphism of functors (see \cite[p. 78]{Ser})
\begin{equation}\label{E:modmap}
\varphi_{\pi,b}:h_{\hat{\O}_{B,b}}:=\Hom(\hat{\O}_{B,b}, -) \to \Def \rC
\end{equation}
By taking the differential of $\varphi_{\pi,b}$, we get the \emph{Kodaira-Spencer map} of the deformation $\pi:(\C,\rC)\to (B, b)$ (see \cite[Thm. 2.4.1(iv) and p. 79]{Ser}
\begin{equation}\label{E:KS}
k_{\pi,b}:=d \varphi_{\pi,b}:T_b(B) \to T\Def \rC=\Ext^1(\Omega_{\rC},\O_{\rC}). 
\end{equation}
Composing with the differential of the morphism \eqref{E:mor-func}, we get the \emph{local Kodaira-Spencer map}
\begin{equation}\label{E:KSloc}
k^{\rm loc}_{\pi,b}:T_b(B) \stackrel{k_{\pi, b}}{\longrightarrow} T\Def \rC\longrightarrow T\Def^{\rm loc}_{\rC}=H^0(\rC,{\mathcal Ext}^1(\Omega^1_\rC,\O_{\rC})).
\end{equation}

In the sequel, we will be often dealing with versal deformations of a curve $\rC$ and versal family of curves, which we are now going to define (see \cite[Def. 2.2.6, Def. 2.5.7]{Ser}).

\begin{definition}\label{D:versal}
Let $\pi:\C\to B$ be a family of curves, i.e. a flat and proper morphism of $k$-schemes whose fibers are (reduced) curves.
\begin{enumerate}[(i)]
\item Let $b$ be a $k$-point of $B$ with fiber $\C_b=\rC$. We say that $\pi:\C\to B$ is \emph{versal} at  $b$ (or that $\pi:(\C,\rC)\to (B, b)$ is a \emph{versal deformation} of $\rC$) if the morphism $\varphi_{\pi, b}$ is smooth. 
\item We say that $\pi:\C\to B$ is a \emph{versal family} if it is versal at every $k$-point of $B$. 
\end{enumerate}
\end{definition}

In the following Fact, we collect the well-known properties of versal deformations of curves, that we are going to need in the sequel. 

\begin{fact}\label{F:versal}
Let $\rC$ be a (reduced and projective) curve. 
\begin{enumerate}[(i)]
\item \label{F:versal1} There exists a versal projective deformation $\pi:(\C,\rC)\to (B,b)$ of $\rC$ over a connected $k$-variety  $B$  (i.e. a scheme of finite type over $k$). 
\item \label{F:versal2} Any versal deformation $\pi:(\C,\rC)\to (B, b)$ of $\rC$ over a scheme $B$ of finite type over $k$ is versal over an open subset of $B$ containing $b$. 
\item \label{F:versal3} Let $\pi:(\C,\rC)\to (B, b)$ be a deformation of $\rC$. Then $\pi:(\C,\rC)\to (B, b)$ is a versal deformation of $\rC$
 and $\Def_\rC$ is smooth if and only if $B$ is smooth at $b$ and the Kodaira-Spencer map $k_{\pi,b}$ is surjective.
 \end{enumerate} 
\end{fact}
It follows from part \eqref{F:versal3} (and what said above) that a deformation $\pi:(\C,\rC)\to (B, b)$  of curve $\rC$ with locally planar singularities (or, more generally, with locally complete intersection singularities) is versal if and only if $B$ is smooth at $b$ and the local Kodaira-Spencer map 
$k_{\pi,b}^{\rm loc}$ is surjective. In particular, if $\pi:\C\to B$ is a versal family of curves with  locally complete intersection singularities, then the base $B$ of the family is smooth.
\begin{proof}
Part \eqref{F:versal1} follows by combining the Schlessinger's criterion for the existence of a versal formal deformation of projective schemes  (see \cite[Cor. 2.4.2]{Ser}),  the Grothedieck's theorem on the effectivity of formal deformations (which uses that $H^2(\rC,\O_\rC)=0$, see \cite[Thm. 2.5.13]{Ser}), and the Artin's theorem on the algebraization of effective formal deformations of projective schemes (see \cite[Thm. 2.5.14]{Ser}).

Part \eqref{F:versal2} is the so called openness of versality (see \cite{Fle}).

Part \eqref{F:versal3} follows from \cite[Prop. 2.5.8(ii)]{Ser}.
\end{proof}

Given a versal family of curves $\pi:\C\to B$, the base scheme $B$ admits a  stratification (called the \emph{equigeneric stratification})  into locally closed subsets according to the cogenus of the geometric fibers of the family $\pi$. More precisely, using the notation introduced in \ref{N:curves}, 
consider the cogenus function
\begin{equation}
\begin{aligned}
\delta: B & \longrightarrow \bbN,\\
 t & \mapsto \delta(\C_{\ov t}),\\
\end{aligned}
\end{equation}
where $\C_{\ov t} :=\pi^{-1}(t)\times_{k(t)} \ov{k(t)}$ is a geometric fiber of $\pi$ over the point $t\in B$.

We call the strata of constant cogenus the {\em equigeneric strata}, and write for any $d\geq 0$
\begin{eqnarray}\label{E:strata-RC}
B^{\delta= d} & := & \{t\in B \: : \: \delta(\C_{\ov t})= d \} \\
 B^{\delta\geq d}& := & \{t\in B \: : \: \delta(\C_{\ov t})\geq d \}
\end{eqnarray}
By the upper semicontinuity of $\delta$ (see \ref{invaria}), we have $B^{\delta\geq d} = \overline{B^{\delta= d}}$. 

The main properties of the equigeneric strata for versal family of curves with locally planar singularities are contained in the following result, due originally  to Teissier and Diaz-Harris if $k=\bbC$ (see \cite[Chap. II]{GLS}), and subsequently  extended to fields of big characteristics  in \cite[Prop. 3.5]{MY} and then to fields of arbitrary characteristics in \cite[Thm. 3.3]{MRV2}.

\begin{fact}\label{F:strati}
Let $\pi:\C\to B$ be a versal family of curves with locally planar singularities. Then we have that (for any $d\geq 0$)
\begin{enumerate}[(i)]
\item \label{T:strati1} the closed subset $B^{\delta\geq d} \subseteq B$ has codimension at least equal to $d$;
\item \label{T:strati2} each generic point $\eta$ of $B^{\delta\geq d}$  is such that $\C_{\ov \eta}$ is a nodal curve.
\end{enumerate}
\end{fact}

On the normalization of each equigeneric stratum of $B$, the pull-back of the  family $\pi:\C\to B$ admits a simultaneous normalization.
More precisely we have the following result which was originally proved in \cite[1.3.2]{T} if $k=\bbC$ and then extended to  arbitrary fields in \cite[Prop. A.2.1]{L-jc}. 

\begin{fact}\label{F:sim-res}
Let $\pi:\C\to B$ be a versal family of curves with locally planar singularities.  
For any $d\geq 0$, consider the normalization $\wt{B^{\delta= d}}$ of the equigeneric stratum with cogenus $d$ and denote by $\pi^d:\C^{\delta=d}\to \wt{B^{\delta= d}}$ the pull-back of the universal family 
$\pi:\C\to B$. Then the normalization $\nu^d:\wt{\C^{\delta=d}}\to \C^{\delta=d}$ is a simultaneous normalization of the family $\pi^d$, i.e.
\begin{enumerate}[(i)]
\item  the composition $\nu^d:\wt{\C^{\delta=d}}\stackrel{\nu^d}{\longrightarrow} \C^{\delta=d}\stackrel{\pi^d}{\longrightarrow} \wt{B^{\delta= d}}$ is smooth;
\item the morphism $\nu^d$ induces the normalization morphism on each geometric fiber of $\pi^d$.
\end{enumerate}
\end{fact}

\subsection{Fine compactified Jacobians} \label{sec:fine}

We collect results on fine compactified Jacobians 
of connected (reduced projective) curves with locally planar singularities and their families.

\subsubsection{Fine compactified Jacobians}

Throughout this subsubsection, we fix a connected (geometrically reduced and projective) curve $\rC$ over a field $k$ and we set $\ov \rC:=C\otimes_{k} \ov k$. Moreover, given a sheaf $\I$ on $\rC$, we denote by $\ov \I$ its pull-back to $\ov \rC$.

Fine compactified Jacobians of $\rC$ will parametrize certain sheaves on $\rC$, which we now introduce.

\begin{definition}
A coherent sheaf $\I$ on a curve $\rC$ is said to be:
\begin{enumerate}[(i)]
\item \emph{rank-$1$} if $\ov \I$ has generic rank $1$ at every irreducible component of $\ov \rC$;
\item \emph{torsion-free} (or pure of dimension one) if  \,${\Supp}(\ov\I)=\ov\rC$ and every non-zero subsheaf $\mathcal{J}\subseteq \I$ is such that $\dim \Supp(\mathcal{J})=1$.
\end{enumerate}
\end{definition}

\noindent Note that any line bundle on $\rC$ is a rank-$1$, torsion-free sheaf.

The construction of fine compactified Jacobians of a reducible curve $\rC$ will depend on the choice of a general polarization on $\rC$, which we now introduce.  
We follow the notation of \cite{MRV1}. 

\begin{definition}\label{pola-def}
\noindent
\begin{enumerate}[(i)]
\item A \emph{polarization} on a curve $\rC$ is a collection  of rational numbers $\pol=\{\pol_{\rC_i}\}$, one for each irreducible component $\rC_i$ of $\ov \rC$,
such that $| \pol |:=\sum_i \pol_{\rC_i}\in \Z$. We call $|\pol |$ the total degree of $\pol$.
Given any subcurve $D \subseteq \ov \rC$, we set $\displaystyle \pol_D:=\sum_{\rC_i\subseteq D} \pol_{\rC_i}$.

\item A polarization $\pol$ is called \emph{integral} at a subcurve $D\subseteq \ov \rC$ if
$\pol_E\in \Z$
 for any connected component $E$ of $D$ and of $D^c$.
A polarization is called
\emph{general} if it is not integral at any non-trivial subcurve $D\subset \ov \rC$.
\end{enumerate}
\end{definition}

Given a polarization $\pol$ on $\rC$, we can define a (semi)stability condition for torsion-free, rank-$1$ sheaves on $\rC$. To this aim,
for each subcurve $D$ of $\ov \rC$ and each torsion-free, rank-$1$ sheaf $\I$ on $\rC$, we denote by $\ov\I_D$ the quotient of the restriction $\ov\I_{|D}$ of $\ov\I$ to $D$ modulo its biggest torsion subsheaf.
It is easily seen that $\ov\I_D$ is torsion-free, rank-$1$ sheaf on $D$.

\begin{definition}\label{D:sh-ss}
\noindent Let $\pol$ be a polarization on $\rC$.
Let $\I$ be a torsion-free rank-1  sheaf on $\rC$ of degree $d=|\pol|$.
\begin{enumerate}[(i)]
\item \label{D:sh-ss1} We say that $\I$ is \emph{semistable} with respect to $\pol$ (or $\pol$-semistable) if
for every non-trivial subcurve $D\subset \ov\rC$, we have that
\begin{equation}\label{E:mdeg-sh}
\chi(\ov\I_D)\geq \pol_D,
\end{equation}
where $\chi$ denotes the Euler-Poincar\'e characteristic.

\item \label{D:sh-ss2} We say that $\I$ is \emph{stable} with respect to $\pol$ (or $\pol$-stable) if it is semistable with respect to $\pol$
and if the inequality (\ref{E:mdeg-sh}) is always strict.
\end{enumerate}
\end{definition}

General polarizations on $\rC$ can be also characterized more geometrically: 

\begin{lemma}\label{L:nondeg}
 \cite[Lemmas 2.14, 5.13]{MRV1}
Let $\pol$ be a polarization on a curve $\rC$. If $\pol$ is general then
every rank-1 torsion-free sheaf which is $\pol$-semistable is also $\pol$-stable. 
The converse implication is true if $\ov\rC$ has locally planar singularities.
\end{lemma}

Fine compactified Jacobians were constructed in full generality by Esteves in \cite{est1}.

\begin{theorem}[Esteves]\label{F:Est-Jac}
Let $\rC$ be a geometrically connected curve and $\pol$ be a general polarization on $\rC$.
There exists a projective scheme $\ov{J}_{\rC}(\pol)$, called  the \emph{fine compactified Jacobian of $\rC$ with respect to the polarization $\pol$}, which is a fine moduli space for torsion-free, rank-$1$, $\pol$-semistable sheaves on $\rC$.
\end{theorem}

Since $\pol$ is general, sheaves in $\ov{J}_\rC(\pol)$ are $\pol$-stable, hence geometrically simple, 
by Lemma \ref{L:nondeg}. This is the reason why $\ov{J}_\rC(\pol)$ is a fine moduli scheme.
Observe also that, clearly, we have that $\ov{J}_{\rC}(\pol)\otimes_k \ov k\cong \ov{J}_{\ov \rC}(\pol)$.

We denote by $J_\rC(\pol)$ the open subset of $\ov{J}_\rC(\pol)$ parametrizing
line bundles on $\rC$. Note that $J_\rC(\pol)$ is isomorphic to the disjoint union of a certain number of copies of the generalized Jacobian $J_\rC=\Pic^{\un 0}(\rC)$ of $\rC$.

If $\rC$ has locally planar singularities and $k=\ov k$, its fine compactified Jacobians enjoy the following properties (see \cite[Thm. A]{MRV1}).

\begin{theorem}\label{F:Jacpl}
Let $\rC$ be a connected curve with locally planar singularities over $k=\ov k$ and $\pol$ a general polarization on $\rC$. Then
\begin{enumerate}[(i)]
\item \label{F:Jacpl1} $\ov{J}_\rC(\pol)$ is a connected reduced projective scheme with locally complete intersection singularities and trivial dualizing sheaf.
\item \label{F:Jacpl2} $J_\rC(\pol)$ is the smooth locus of $\ov{J}_\rC(\pol)$. In particular, $J_\rC(\pol)$ is dense in $\ov{J}_\rC(\pol)$ and $\ov{J}_\rC(\pol)$ has pure dimension equal to the arithmetic genus $g(\rC)$ of $\rC$.
\item \label{F:Jacpl3} The number of irreducible components of $\ov{J}_\rC(\pol)$ depends only on the curve $\rC$ and not on the polarization $\pol$.
\end{enumerate}
\end{theorem}

Therefore, the number of irreducible component of any fine compactified Jacobian of a connected curve $\rC$ with locally planar singularities over $k=\ov k$ is an invariant of $\rC$, which is usually called the complexity of $\rC$ and denoted by $c(\rC)$.
We refer the reader to \cite[Sec. 5.1]{MRV1} for an explicit formula for $c(\rC)$ in terms of the intersection numbers between the subcurves of $\rC$. We just mention that if $\rC$ is nodal, then $c(\rC)$ is given by the complexity of its dual graph, i.e. the number of its spanning trees.

The above Theorem \eqref{F:Jacpl} implies that any two fine compactified Jacobians of a curve $\rC$ with locally planar singularities over $k=\ov k$ are birational Calabi-Yau (singular) varieties.
However, in \cite[Sec. 3]{MRV1}, the authors constructed some nodal reducible curves which do have non isomorphic (and even non homeomorphic if $k=\CC$) fine compactified Jacobians. Despite this, Theorem \ref{thm:jac} implies that any two fine compactified Jacobians of a  curve $\rC$ with locally planar singularities have the same Betti numbers if $k=\CC$, recovering
in particular Theorem \ref{F:Jacpl}\eqref{F:Jacpl3}. 
It is shown in \cite{MRV2} and \cite{MRV3} that all fine compactified Jacobians are derived equivalent.

\vspace{0,2cm}
\subsubsection{Relative fine compactified Jacobians} \label{S:famJac}

Given a projective family $\pi:\C\to B$ of geometrically connected  (and geometrically reduced) curves, i.e. a projective and flat morphism $\pi$ whose geometric fiber $\cC_{\ov b}:=\pi^{-1}(b)\otimes_{k(b)} \ov{k(b)}$ over any point $b\in B$ 
is a connected (and reduced)  curve,  a \emph{relative fine compactified Jacobian} for $\pi$ is a scheme $\pi^J:\ov{J}_{\cC}\to B$ projective over $B$, 
such that the geometric fiber $(\ov{J}_{\cC})_{\ov b}:=(\pi^J)^{-1}(b)\otimes_{k(b)} \ov{k(b)}$ over any point $b\in B$ is a fine compactified Jacobian for the curve $\cC_{\ov b}$.

In the sequel, we will need the following result of the existence  of relative fine compactified Jacobians for families  of  geometrically connected (geometrically reduced and projective)  curves.

\begin{theorem}\label{T:fam-Jac}
Let $\pi:\C\to B$ be a projective family of geometrically connected curves. 
\begin{enumerate}
\item Up to passing to an \'etale cover of $B$,
there exists a relative fine compactified Jacobian $\pi^J:\ov{J}_{\cC}\to B$ for $\pi$.
\item Fix a point $b\in B$ and a general polarization $\pol$ on the  fiber
$\C_{b}$ over $b$. Then, up to replacing $B$ with an \'etale neighborhood of $b$,
there exists a family of fine compactified Jacobians $\pi^{J}:\ov{J}_{\C}(\pol)\to B$ such that $\ov{J}_{\C}(\pol)_{b}=\ov{J}_{\C_b}(\pol)$. Moreover we have (up to replacing $B$ with an open neighborhood of $b$):
\begin{enumerate}[(i)]
\item if $\C_{\ov b}$ has locally planar singularities and $B$ is geometrically unibranch (e.g. normal) and reduced at $b$, then $\pi^J$ is flat with geometric fibers of pure dimension $g(\C_{\ov b})$;
\item if $\C_{\ov b}$ has locally planar singularities and $\pi$ is versal at $b$, then $\ov{J_\C}(\pol)$ is regular.
\end{enumerate}
\end{enumerate}
\end{theorem}
\begin{proof}
The proof is similar to the one of \cite[Thm. 5.4, Thm. 5.5]{MRV1} (which deals with the effective semiuniversal deformation family of a curve $\rC$), building upon the work of Esteves \cite{est1}. We omit the details.
\end{proof}




\section{Nodal curves}\label{nodal_curves_section}

In this section we express the counting function of the Hilbert scheme of a nodal curve defined over a finite field
as a sum of trace-functions of Cattani-Kaplan-Schmid complexes. This is the most important step in the proof of Theorem \ref{thm:hilb}.

Throughout this section, we always consider the following 
\begin{setup}\label{setup}
Let  $\rC_o$ be a {\em nodal } curve defined over a finite field $k:=\FF_{\pi}$ and 
$\Gamma=\Gamma_{\rC}$ is the dual graph of $\rC=\rC_o\times_{\FF_{\pi}} \ov{\FF_\pi}$.
Let $\pi_o: \cC_o \to B_o$ be a versal family of nodal curves with central fibre the curve $\rC_o = {\cC_o}_b$ and assume, 
up to localizing at $b$,  that $B_o$ is smooth and irreducible. Denote by $\pi: \cC \to B$ the base change of the family $\pi_o$ to the algebraic closure $\ov k=\ov{\FF_\pi}$.
The discriminant locus $ \Delta $ of $\pi$ is a normal crossing divisor on $B$ which has a component ${\Delta}_e$ for each node $e$ of ${\rC}$. We set $\Breg:=B\setminus \Delta$. 

Sometimes we will need to assume that the cardinality of the base field $\FF_{\pi}$ is big enough (compared to the cogenus $\delta(\rC)$ of $\rC$), which is enough for our applications since the families $\pi_o:\C_o\to B_o$ we will be considering arise from the reduction of families defined over the complex numbers. 

\end{setup}

\subsubsection{The dual graph}\label{S:dualgr}
We write $\Gamma = \Gamma_{\rC}$ for the dual graph of the curve ${ \rC}$:
its vertices $v \in \rV$ correspond to the irreducible
components
of ${\rC}$, and its edges $e \in \rE$ correspond to the nodes of ${\rC}$. We will also be considering the set $\OE$ of oriented edges of $\Gamma$ and we will denote by  $\er$ and $\el$  the two oriented 
edges corresponding to an (unoriented) edge $e$ of $\Gamma$. 
Note that, since we do not assume $\rC_o$ geometrically connected, $\Gamma$ may be disconnected.

\smallskip
 
The Galois group $\Gal(\overline{k}/k)$, which is topologically generated by Frobenius, 
acts on the graph $\Gamma$, and in particular on the sets $\OE$ and $\rV$. The action of Frobenius on the vertex set $\rV$ corresponds to the action of Frobenius on the irreducible components of $\rC$. 
The action of Frobenius on the set $\OE$ of oriented edges is determined by the types of the nodes of $\rC_o$ as we now explain. 
A node of $\rC_o$  is identified by one integer $r$ and one ``sign" $\epsilon=\pm 1$. By this we mean that: 
\begin{enumerate}
\item 
({\bf The split case.})
$(r,+)$ is analytically isomorphic to $\mathrm{Spec}\, \FF_{\pi^r}[[X,Y]]/(X^2-Y^2)$ as a $\FF_{\pi}$ scheme, i.e. the point correspond to $r$  geometric points with rational tangents. In this case the normalization is 
$\mathrm{Spec}\, \left(\FF_{\pi^r}[[X]] \times \FF_{\pi^r}[[Y]]\right)= \mathrm{Spec}\, \left(\FF_{\pi^r}[[X]]\right)\coprod \mathrm{Spec}\, \left(\FF_{\pi^r}[[Y]]\right). $
\item
({\bf The non-split case.})
$(r,-)$ is analytically isomorphic to $\mathrm{Spec}\, \FF_{\pi^r}[[X,Y]]/(X^2-aY^2)$ as a $\FF_{\pi}$ scheme, with $a \notin \FF_{\pi^r}^2$ i.e. the point correspond to $r$  geometric points with non rational tangents (a further quadratic extension is needed). In this case the normalization is  $\mathrm{Spec}\, \left(\FF_{\pi^{2r}}[[X]]\right).$
\end{enumerate}

Frobenius  acts on the set of $2r$ oriented edges $\{\er_1, \cdots , \er_r,\el_1, \cdots , \el_r \}$ corresponding to the $r$ nodes of $\rC$ that lie over the node of $\rC_o$: 
in the first case, one can number and orient the edges so that $Fr(\er_i)=\er_{i+1}$ for $i<r$ and  $Fr(\er_r)=\er_{1}$, and similarly with the $\el_i$'s so that there are two orbits of $r$ elements each,  
whereas in the second case $Fr(\er_i)=\er_{i+1}$ for $i<r$,  $Fr(\er_r)=\el_{1}$ and $Fr(\el_i)=\el_{i+1}$, so that there is just one orbit.

We write $\VV = \VV_\Gamma := C_0(\Gamma, \Qlbar)$ and $\EE = \EE_\Gamma := C_1(\Gamma, \Qlbar)$ for the $\Gal(\overline{k}/k)$-modules of zero- and one-simplicial chains on $\Gamma$. Explicitly, $\VV$ is the $\Qlbar$-vector space
of $\Qlbar$-linear combination  of vertices of $\Gamma$ and $\EE$ is the $\Qlbar$-vector space of $\Qlbar$-linear combination of oriented edges of $\Gamma$ modulo the relation  $\er=-\el$, where $\er$ and $\el$ denote the two oriented 
edges corresponding to an (unoriented) edge $e$ of $\Gamma$. 
The actions of $\Gal(\overline{k}/k)$ on  $\VV$ and $\EE$ are induced by the action on $V$ and $\OE$ so that $\VV$ is a permutation representation while $\EE$ is only a signed permutation representation (because the Galois action can reverse the oriented edges of $\Gamma$, as explained above). The homology of the graph $\Gamma$ is defined via the following exact sequence 
\begin{equation}\label{E:seq}
 0 \to H_1(\Gamma, \Qlbar)  \to \EE  \stackrel{\partial}{\longrightarrow} \VV \to H_0(\Gamma, \Qlbar)  \to 0,
 \end{equation}
where $\partial$ is the boundary map which sends an oriented edge into the difference between its target and its source. 

We write $\VV^* = C^0(\Gamma, \Qlbar)$ and $\EE^* = C^1(\Gamma, \Qlbar)$
for the dual $\Gal(\overline{k}/k)$-modules of zero- and one-simplicial cochains on $\Gamma$. 
Since $\VV$ and $\EE$ are both  signed permutation representations, there are isomorphisms
of $\Gal(\overline{k}/k)$-modules $\EE \cong \EE^*$ and $\VV \cong \VV^*$.  The cohomology of $\Gamma$ is defined by mean of the following exact sequence 
\begin{equation}\label{E:cohseq}
0 \to H^0(\Gamma, \Qlbar) \to \VV^* \stackrel{\partial^*}{\longrightarrow} \EE^* \to H^1(\Gamma, \Qlbar) \to 0,
\end{equation}
where $\partial^*$ is the dual of the map $\partial$.

Since the 
$\Gal(\overline{k}/k)$ action on  $\EE$, $\VV$, $\EE^*$, $\VV^*$, $H_0(\Gamma, \Qlbar)$, $H_1(\Gamma, \Qlbar)$, 
$H^0(\Gamma, \Qlbar)$, $H^1(\Gamma, \Qlbar)$ factors through a finite group, all these spaces are pure of weight zero.

\subsubsection{Geometric interpretation of the cohomology of the dual graph.}\label{dual_graph}

The homology and cohomology groups of the dual graph $\Gamma$ of $C$ arise geometrically from curves related to $\rC$ by normalization and deformation.

Cohomology of the graph $\Gamma$ comes from the normalization $\nu: \rC^\nu \to  \rC$.  The sequence of sheaves
$$0 \to \Qlbar \to \nu_* \Qlbar \to \nu_* \Qlbar/ \Qlbar \to 0$$
yields by taking cohomology:
$$0 \to H^0(\rC, \Qlbar) \to H^0(\rC^\nu, \Qlbar) \to
H^0(\rC, \nu_* \Qlbar/ \Qlbar)
\to H^1(\rC, \Qlbar) \to H^1(\rC^\nu, \Qlbar)  \to 0.$$
We have defined $\VV^*, \EE^*$ so as to have canonical, 
$\Gal(\overline{k}/k)$-equivariant identifications
\begin{eqnarray*}
\VV^* & = & H^0(\rC^\nu, \Qlbar), \\
\EE^* & = & H^0(\rC, \nu_* \Qlbar/ \Qlbar) 
\end{eqnarray*}

Substituting in $H^1(\Gamma, \Qlbar) = \mathrm{Cok}(H^0(\rC^\nu, \Qlbar) \to 
H^0(\rC, \nu_* \Qlbar/ \Qlbar))$, we find the short exact sequence
\begin{equation}
0 \to H^1(\Gamma, \Qlbar) \to H^1(\rC, \Qlbar) \to H^1(\rC^\nu, \Qlbar) \to 0,
\end{equation}
which, since $H^1(\Gamma, \Qlbar)$ is pure of weight zero and $H^1(\rC^\nu, \Qlbar)$ is pure of weight one,
gives the weight filtration of $H^1(\rC, \Qlbar)$.

On the other hand, homology of the graph comes from a one-parameter smoothing $\sigma: \mathcal{C} \to \DD$ of $\rC$,  with special fibre $\mathcal{C}_0 ={\rC}$ and
geometric generic fibre $\mathcal{C}_{\overline{\eta}}$.  The cohomology of the nearby-vanishing sequence gives:

\begin{equation}\label{PsiPhi}
0 \to H^1(\rC, \Qlbar) \to H^1 ( \mathcal{C}_{\overline{\eta}}, \Qlbar) \to H^1(\rC, \Phi_\sigma \Qlbar) \to H^2(\rC, \Qlbar) \to
H^2 ( \mathcal{C}_{\overline{\eta}}, \Qlbar) \to 0.
\end{equation}
By Poincar\'e duality we have
$$
H^2(\rC, \Qlbar) = H^2(\rC^\nu, \Qlbar) \cong H^0(\rC^\nu, \Qlbar)^*\otimes \LL  = \VV \otimes \LL
$$
and, likewise
$$H^2 ( \mathcal{C}_{\overline{\eta}}, \Qlbar) \cong H^0 ( \mathcal{C}_{\overline{\eta}}, \Qlbar)^*\otimes \LL  \cong
H^0 ( \rC, \Qlbar)^* \otimes \LL.$$
Finally, we have by Picard-Lefschetz formula (\cite{Mi}, p.207):
$$
H^1(\rC, \Phi_\sigma \Qlbar) \cong \EE \otimes \LL.
$$
Substituting in formula \eqref{PsiPhi} we find:
\begin{equation}
 0 \to H^1(\rC, \Qlbar) \to H^1 ( \mathcal{C}_{\overline{\eta}}, \Qlbar) \to H_1(\Gamma, \Qlbar)\otimes \LL \to 0. 
 \end{equation}
The (monodromy-)weight filtration on $H^1 ( \mathcal{C}_{\overline{\eta}}, \Qlbar) $ is:
\begin{eqnarray*}
W_0 H^1 ( \mathcal{C}_{\overline{\eta}}, \Qlbar) & = & H^1(\Gamma, \Qlbar), \\
W_1 H^1 ( \mathcal{C}_{\overline{\eta}}, \Qlbar) & = & H^1(\rC, \Qlbar), \\
 W_2 H^1 ( \mathcal{C}_{\overline{\eta}}, \Qlbar) & = & H^1 ( \mathcal{C}_{\overline{\eta}}, \Qlbar),
\end{eqnarray*}
with associated graded pieces
\begin{eqnarray*}
Gr^W_0 H^1 ( \mathcal{C}_{\overline{\eta}}, \Qlbar) & = & H^1(\Gamma, \Qlbar), \\
Gr^W_1 H^1 ( \mathcal{C}_{\overline{\eta}}, \Qlbar) & = & H^1(\rC^\nu, \Qlbar), \\
 Gr^W_2 H^1 ( \mathcal{C}_{\overline{\eta}}, \Qlbar) & = & H_1 (\Gamma, \Qlbar)\otimes \LL.
\end{eqnarray*}

\subsubsection{Subgraphs and partial normalizations.}\label{S:subgr}

For every subset $I \subset \rE$, we factor the normalization map
\[
\rC^\nu \stackrel{\nu^I}{\to} \rC^I \stackrel{\nu_I}{\to} \rC
\]
where $\nu_I: \rC^I \to \rC$ is the partial normalization of the nodes of the subset $I$,
and $\nu^I: \rC^\nu \to \rC^I$ for the remaining normalization.

 We have sequences
$$0 \to H^0(\rC, \Qlbar) \to H^0(\rC^I, \Qlbar) \to
H^0(\rC, \nu_{I*} \Qlbar/ \Qlbar)
\to H^1(\rC, \Qlbar) \to H^1(\rC^I, \Qlbar)  \to 0$$
and
$$0 \to H^0({\rC}^I, \Qlbar) \to H^0(\rC^\nu, \Qlbar) \to
H^0({\rC}^I, \nu^I_{*} \Qlbar/ \Qlbar)
\to H^1({\rC}^I, \Qlbar) \to H^1({\rC}^\nu, \Qlbar)  \to 0.$$
The dual graph of the partial normalization ${\rC}^I$ is the graph $\Gamma\setminus I$, which is obtained from $\Gamma=\Gamma_{\rC}$ by deleting the edges corresponding to $I$.
As in \S\ref{dual_graph}, we have canonical identifications $\EE_{\Gamma \setminus I}^*= H^0({\rC}^I, \nu^I_{*} \Qlbar/ \Qlbar)$ and
$\VV^*_{\Gamma \setminus I}=H^0({\rC}^\nu, \Qlbar)=\VV^*_\Gamma=\VV^*$.
Moreover, we set $\EE_I^* = H^0({\rC}, \nu_{I*} \Qlbar/ \Qlbar)$ so that we
have a canonical splitting $\EE^*=\EE_\Gamma^* = \EE_{\Gamma \setminus I}^* \oplus \EE_I^*$.

We now introduce a collection of subsets of $\rE$ which will play an important role in what follows.

\begin{definition}\label{D:defC}
We write $\mathscr{C}(\Gamma)$ for the collection of subsets of $\rE$ whose removal disconnects no component
of $\Gamma$, i.e. a subset $I\subseteq \rE$ belongs to $\mathscr{C}(\Gamma)$ if and only if $\Gamma\setminus I$ has the same number of connected components of $\Gamma$.

 We set $n_i(\Gamma) := \# \{I \in \mathscr{C}(\Gamma)\, | \, \dim H_1(\Gamma \setminus I) = i \}$.
\end{definition}

Note that $n_0(\Gamma)$, i.e. the cardinality of the set of maximal elements of $\mathscr{C}(\Gamma)$, is also equal to the complexity $c(\Gamma)$ of $\Gamma$, i.e. the number of spanning forests of $\Gamma$.

An alternative characterization of the elements of $\mathscr{C}(\Gamma)$ is provided by the following

\begin{lemma}\label{L:setC}
A subset $I\subseteq \rE$ belongs to  $\mathscr{C}(\Gamma)$ iff the composition $\EE_I^*\to \EE^* \to H^1(\Gamma, \Qlbar)$ is injective.  In that case,
the following sequence is exact:
$$0 \to \EE_I^* \to H^1(\Gamma, \Qlbar) \to H^1(\Gamma \setminus I, \Qlbar) \to 0.$$
\end{lemma}
\begin{proof}
The inclusion of graphs $\Gamma\setminus I\hookrightarrow \Gamma$ induces a pull-back map from the sequence \eqref{E:cohseq} to the analogous sequence for $\Gamma\setminus I$.
Applying the snake lemma to this map of sequences and using that $\VV_{\Gamma\setminus I}^*=\VV_{\Gamma}^*$, we get the exact sequence
$$0\to H^0(\Gamma,\Qlbar)\to H^0(\Gamma\setminus I, \Qlbar) \to \EE_I^*\to H^1(\Gamma,\Qlbar)\to H^1(\Gamma\setminus I,\Qlbar) \to 0.$$
By Definition \ref{D:defC}, the subset $I$ belongs to $\mathscr{C}(\Gamma)$ if and only if the map $H^0(\Gamma,\Qlbar)\to H^0(\Gamma\setminus I, \Qlbar)$ is an isomorphism.
By the above exact sequence, this happens precisely when the map $\EE_I^*\to H^1(\Gamma, \Qlbar)$ is injective and in that case we get the required short exact sequence.
\end{proof}

\begin{remark}
It follows from Lemma \ref{L:setC} that  $\mathscr{C}(\Gamma)$ is the collection of all subsets of $\rE$ whose images under the map
$\EE^* \to H^1(\Gamma,\Qlbar)$ remain linearly independent.
Thus $\mathscr{C}(\Gamma)$ is the collection of independent elements of a (representable)
matroid -- in particular, a simplicial complex -- which is usually called the cographic  matroid of the graph $\Gamma$.
\end{remark}

Fixing orientations of each edge $e \in \rE$ of $\Gamma$ and an ordering on $\rE$
determines, for all $I \subset \rE$, `volume' elements $e_I^* \in \wedge^{|I|} \EE_I^*$, well defined up to a sign. 
Lemma \ref{L:setC} may be reformulated as the assertion that $I \in  \mathscr{C}(\Gamma)$ if and only if the image of $e_I^*$ in $\wedge^{|I|} H^1(\Gamma,\Qlbar)$ is non-zero. Indeed, even more is true as the following Lemma shows.

\begin{lemma}\label{L:mapeI}
If $I \in \mathscr{C}(\Gamma)$, there is an injective map, well-defined up to a sign,
$$\wedge e_I^* : \bigwedge^{i - |I|} H^1(\Gamma \setminus I, \Qlbar) \to \bigwedge^i  H^1(\Gamma, \Qlbar).$$
\end{lemma}
\begin{proof}
The map is defined by lifting $\eta \in \bigwedge^{i - |I|} H^1(\Gamma \setminus I, \Qlbar)$ arbitrarily
to an element in $\bigwedge^{i - |I|} H^1(\Gamma,  \Qlbar),$ and then wedging by $e_I^*$.  This
is well defined because the ambiguity in the lift is killed by $\wedge e_I^*$.
\end{proof}

\subsection{Determination of $IC(\Lambda^i R^1 \pi_{sm *} \Qlbar )$} \label{sec:CKS}

In the setup \ref{setup}, consider  the  local system ${\mathscr V}^1:=R^1\pi_*{\Qlbar}_{|\Breg}$ on $\Breg$, which, defines (see \S \ref{sec:CKS_preliminaries}) a local system ${\mathscr V}^1_{\bigcap_e \Delta_e}:=\Psi({\mathscr V}^1)$ on $\bigcap_e \Delta_e \ni b$, endowed with $|E|$ commuting
twisted nilpotent endomorphisms
$$N_e : {\mathscr V}^1_{\bigcap_e \Delta_e}\to {\mathscr V}^1_{\bigcap_e \Delta_e}\otimes \LL .$$ 
We also have the local systems ${\mathscr V}^i:=\bigwedge^i {\mathscr V}^1$, and corresponding sheaves ${\mathscr V}^i_{\bigcap_e \Delta_e}:=\Psi({\mathscr V}^i)$ on $\bigcap_e \Delta_e$,
endowed with commuting twisted nilpotent endomorphisms
\[
N_e^{(i)}: {\mathscr V}^i_{\bigcap_e \Delta_e} \to {\mathscr V}^i_{\bigcap_e \Delta_e}\otimes \LL.
\]
It is known that the local system ${\mathscr V}^1$, and therefore also its exterior powers  ${\mathscr V}^i$, are tamely ramified \cite[Thm. 1.5]{Ab}.
As we are interested in pointwise computations, we may consider a normal slice so we assume
$\bigcap_e \Delta_e=\{b\}$ and identify ${\mathscr V}^1_{\bigcap_e \Delta_e}\cong H^1(\mathcal{C}_{\overline{\eta}},\Qlbar)$, where $\mathcal{C}_{\overline{\eta}}$ is a geometric generic fiber of a one-parameter smoothing of $\rC$. Remark that the monodromy filtration is independent of the one-parameter  smoothing that we choose as it coincides with
the weight filtration. The monodromy-weight filtration of ${\mathscr V}^1_{\bigcap_e \Delta_e}$ is hence identified with that of $H^1(\mathcal{C}_{\overline{\eta}},\Qlbar)$ described in \S \ref{dual_graph}.

It follows immediately from weights considerations that the map
$$N_e: H^1(\mathcal{C}_{\overline{\eta}},\Qlbar) \to H^1(\mathcal{C}_{\overline{\eta}},\Qlbar)\otimes \LL $$
factors as
\begin{equation}\label{E:facto}
H_1 (\Gamma, \Qlbar)\otimes \LL= Gr^W_2 H^1 ( \mathcal{C}_{\overline{\eta}}, \Qlbar)  \to    Gr^W_0 H^1 ( \mathcal{C}_{\overline{\eta}}, \Qlbar)\otimes \LL =  H^1(\Gamma, \Qlbar)\otimes \LL,
\end{equation}
and it is easily seen to be given by
\begin{equation}\label{E:mapNe}
H_1(\Gamma, \Qlbar) \hookrightarrow \EE \xrightarrow{t \mapsto \langle \vec{e}^*, t \rangle \cdot \vec{e}^*} \EE^*  \twoheadrightarrow H^1(\Gamma, \Qlbar)
\end{equation}
where $\vec{e}$ is an orientation of the edge $e$ and $\vec{e}^*$ is its dual element in  $\EE^*$ (note that the above is independent of the orientation of $e$).
Similarly, for the exterior powers, we have the identification ${\mathscr V}^i_{\bigcap_e \Delta_e}\cong \bigwedge^i H^1(\mathcal{C}_{\overline{\eta}},\Qlbar)$ under which the operators $N_e^{(i)}$ become
\begin{eqnarray*}
N_e^{(i)}= \textstyle\bigwedge^i H^1(\mathcal{C}_{\overline{\eta}}, \Qlbar)
& \to & \textstyle\bigwedge^i H^1(\mathcal{C}_{\overline{\eta}}, \Qlbar)\otimes \LL \\
c_1 \wedge \cdots \wedge c_i & \mapsto & \sum_{k=1}^i c_1 \wedge \cdots \wedge N_e(c_k) \wedge \cdots \wedge c_i .
\end{eqnarray*}

For $I \subset \rE$ we write
$$N_I^{(i)} := \prod_{e \in I} N_e^{(i)}: \textstyle\bigwedge^i H^1(\mathcal{C}_{\overline{\eta}}, \Qlbar) \to \textstyle\bigwedge^i H^1(\mathcal{C}_{\overline{\eta}}, \Qlbar)\otimes \LL^{|I|}.$$

The stalk of $IC(\mathscr{V}^i)$ at $\{b\}=\bigcap_e \Delta_e$ is quasi-isomorphic to the following
complex of continuous $\Qlbar$-representations of $\Gal(\overline{k}/k)$:
\begin{equation}
\label{ckscplx}
 0 \to  \textstyle\bigwedge^i H^1(\mathcal{C}_{\overline{\eta}}, \Qlbar) \to \bigoplus_{I \subseteq E, |I|=1} \im N_I^{(i)} \to \bigoplus_{I \subseteq E, |I|=2} \im N_I^{(i)} \to \cdots
\end{equation}

where the first term
$\bigwedge^i H^1(\mathcal{C}_{\overline{\eta}}, \Qlbar)$ is in homological degree zero. 
Omitting, for brevity of notation, to indicate the nilpotent endomorphisms, we denote this complex by ${\mathbf C}^{\bullet}(\bigwedge^i H^1(\mathcal{C}_{\overline{\eta}}, \Qlbar))$.

We also define operators by restricting the above to the even weight pieces of the associated graded pieces,
$H^1(\mathcal{C}_{\overline{\eta}}, \Qlbar)_{ev} := H^1(\Gamma, \Qlbar) \oplus H_1 (\Gamma, \Qlbar) \otimes \LL$, i.e.,
$$\hat{N}_e: H^1(\Gamma, \Qlbar) \oplus H_1 (\Gamma, \Qlbar) \otimes \LL \to
H^1(\Gamma, \Qlbar) \otimes \LL \oplus H_1 (\Gamma, \Qlbar) \otimes \LL^2,$$
and similarly for the operators $\hat{N}_e^{(i)}$ and  $\hat{N}_I^{(i)}$.

We want now to describe the image of the maps $\hat{N}_I^{(i)}$. Recall from Lemma \ref{L:mapeI} that if $I \in \mathscr{C}(\Gamma)$, then there is an injective map $\wedge e_I^*: \bigwedge^{i-|I|} H^1(\Gamma \setminus I, \Qlbar) \to \bigwedge^i H^1(\Gamma, \Qlbar)$. Using the natural injection $H_1(\Gamma\setminus I, \Qlbar) \hookrightarrow H_1(\Gamma, \Qlbar)$ coming from the inclusion of graphs $\Gamma\setminus I\subset \Gamma$, we get an injective map
\begin{equation}\label{E:wedge}
\wedge e_I^*: \textstyle\bigwedge^{i - |I|} \left( H^1(\Gamma \setminus I, \Qlbar) \oplus H_1(\Gamma \setminus I, \Qlbar)\otimes \LL \right) \hookrightarrow \textstyle\bigwedge^i
\left( H^1(\Gamma, \Qlbar) \oplus H_1(\Gamma, \Qlbar) \otimes \LL \right).
\end{equation}

\begin{lemma} {\bf (the main calculation)}\label{L:main}
The image of $\hat{N}_I^{(i)}$ is zero unless $I \in \mathscr{C}(\Gamma)$, and in this case, it is equal to to the image of the map \eqref{E:wedge} twisted by $\LL^{|I|}$.
\end{lemma}
\begin{proof}
We recall how the choice of a spanning forest of $\Gamma$ (i.e. a spanning tree on each connected component of $\Gamma$)
gives rise to dual bases for $H_1(\Gamma):=H_1(\Gamma, \Qlbar)$ and $H^1(\Gamma):=H^1(\Gamma, \Qlbar)$.
Let $J \subseteq \rE$ be a maximal element of $\mathscr{C}(\Gamma)$ so that $\Gamma \setminus J$ is a spanning forest of $\Gamma$.
Then on one hand, for each $e \in J$, we have
the corresponding $\vec{e}^* \in \EE^*$, and their images in $H^1(\Gamma)$ give a basis.  On the other hand, for each $e\in J$,
there is unique loop in $\Gamma \setminus (J \setminus e)$ which gives rise to an element of $H_1(\Gamma)$ denoted by
 $\underline{e}$; this again gives a basis.  We have $\langle \vec{e}_i^{*}, \underline{e}_j \rangle = \pm \delta_{ij}$ for each $e_i,e_j \in E$. 

We return to the problem at hand.
By induction on $|I|$ and the obvious compatibility of $N_e$ with the analogous operator on the complex associated to a subgraph
$\Gamma \setminus e'$, it suffices to consider the case when $I = \{e\}$.
Let $\Gamma_e$ be the component of $\Gamma$ containing $e$.
If the removal of the edge $e$ disconnects $\Gamma_e$, then certainly no cycle $t \in H_1(\Gamma)$ can
contain the edge $e$, hence $\langle \vec{e}^*, t \rangle = 0$ for any $t$, and so $N_e \equiv 0$.

Otherwise, there exists some maximal $e \in J \in \mathscr{C}(\Gamma)$. Let
$\{\underline{e} = \underline{e}_1, \underline{e}_2, \ldots\}$ and
$\{\vec{e}^* = \vec{e}_1^*, \vec{e}_2^*, \ldots\}$ be the corresponding dual bases.
Observe that $J \setminus e \in \mathscr{C}(\Gamma \setminus e)$ is again maximal, and the resulting dual
basis of $H_1(\Gamma \setminus e)$ and $H^1(\Gamma \setminus e)$ are
$\{\underline{e}_2, \ldots\}$ and $\{\vec{e}_2^*, \ldots\}$.

We compute the action of $\hat{N}_e^{(i)}$:
\begin{eqnarray*}
& & \hat{N}_e^{(i)}(\vec{e}_{a_1}^* \wedge \cdots \wedge \vec{e}_{a_{d}}^* \wedge
\underline{e}_{b_1} \wedge \cdots \wedge \underline{e}_{b_{i-d}} ) \\ & = &
\sum_{r = 1}^{i-d} \vec{e}_{a_1}^* \wedge \cdots \wedge \vec{e}_{a_{d}}^* \wedge
\underline{e}_{b_1} \wedge \cdots \wedge \hat{N}_e(\underline{e}_{b_r}) \wedge \cdots \wedge \underline{e}_{b_{i-d}}
\\
& = & \left( \sum_{r = 1}^{i-d} \pm \delta_{1, b_r} \cdot \vec{e}_{a_1}^* \wedge \cdots \wedge \vec{e}_{a_{d}}^* \wedge
\underline{e}_{b_1} \wedge \cdots \wedge \hat{\underline{e}}_{b_r}\wedge \cdots \wedge \underline{e}_{b_{i-d}} \right) \wedge \vec{e}^*
\end{eqnarray*}
If any of the $a_i = 1$, then this sum vanishes.  In any case, the sum has at most one nonvanishing term, that of $b_r = 1$.
Assuming without loss of generality that $a_1 < a_2 < \cdots$ and $b_1 < b_2 < \cdots$, the sum vanishes unless
$a_1 > 1$ and $b_1 = 1$; and
$$\hat{N}_e^{(i)}(\vec{e}_{a_1 > 1}^* \wedge \cdots \wedge \vec{e}_{a_{d}}^* \wedge
\underline{e} \wedge \underline{e}_{b_2 > 1} \wedge \cdots \wedge \underline{e}_{b_{i-d}}) =
\pm \vec{e}^* \wedge \left( \vec{e}_{a_1 > 1}^* \wedge \cdots \wedge \vec{e}_{a_{d}}^* \wedge \underline{e}_{b_2 > 1} \wedge \cdots \wedge \underline{e}_{b_{i-d}}
\right) $$
This completes the proof.
\end{proof}

\begin{remark}\label{R:vanNI}
In particular, if $i< |I|$ or $h^1(\Gamma)<|I|$ then $\hat{N}_I^{(i)}$ vanishes. This is true also for the map 
$$
N_I^{(i)}: \textstyle\bigwedge^i H^1(\mathcal{C}_{\overline{\eta}}, \Qlbar) \to \textstyle\bigwedge^i H^1(\mathcal{C}_{\overline{\eta}}, \Qlbar)\otimes \LL^{|I|}.
$$
Indeed, if $i< |I|$ then $N_I^{(i)}$ vanishes because the weights of the source go from $0$ to $2i$  while those of the target from $2|I|$ to $2|I| + 2i$. Moreover, if $h^1(\Gamma)<|I|$ then the map $N_I^{(i)}$ vanishes  because
of the factorization \eqref{E:facto}.
\end{remark}

The next corollary follows directly from the previous considerations:

\begin{corollary} \label{cor:weightpolyic}
We have the following evaluations of weight polynomials:
\begin{equation*}\label{weightcks}
\mathfrak{w}\left( \sum_i q^i IC\left(\textstyle\bigwedge^i R^1 \pi_*{\Qlbar}_{|\Breg}\right)_b[-i]\right)
 =
(1+qt)^{{2 g(\ov{\rC}^\nu)}}
\sum_{i\geq 0} n_i (\Gamma) \cdot  (q t^2)^{h^1(\Gamma) - i} \left((1-q t^2) (1-q) \right)^{i}.
\end{equation*}
In particular, setting $q=1$, we get that 
\begin{equation}\label{weight_cks} \mathfrak{w}\left( \sum_i IC\left(\textstyle\bigwedge^i R^1\pi_*{\Qlbar}_{|\Breg}\right)_b[-i]\right) =
(1 + t)^{2 g(\ov{\rC}^\nu)} t^{2  h^1(\Gamma)} c(\Gamma).
\end{equation}
\end{corollary}

\subsection{The Hilbert scheme of a nodal curve} \label{S:hilbert_nodal} 



In this subsection, we will be using the following

\begin{notation}
\noindent 
\begin{enumerate}
\item 
Equalities in this section are in the counting sense, as we now explain. To any element $\sum_i \lambda_i X_{o,i}$ of the Grothendieck ring $K_0(\Var_{\FF_{\pi}})$ of varieties over $\FF_{\pi}$, it is associated the counting function
$$r\in \bbN \mapsto \sum_i \lambda_i X_{o,i}(\FF_{\pi^r}).$$
And to any element $\sum_i \lambda_i W_i$ of the $K$-ring $K_0(\Rep(\Fr))$  of (finite dimensional) $\Qlbar$-vector spaces with an action of Frobenius $\Fr$, it is associated the counting function 
$$r\in \bbN\mapsto \sum_i \lambda_i \Tr(\Fr^r: W^i \to W^i).$$
Two objects belonging to either $K_0(\Var_{\FF_{\pi}})$ or $K_0(\Rep(\Fr))$ are said to be equal if they have the same counting function. And two formal power series in $q$ with coefficients in either $K_0(\Var_{\FF_{\pi}})$ or $K_0(\Rep(\Fr))$ are said to be equal if each of their coefficients has the same counting function.

For example, if $\rC_o$ is a geometrically connected, nonsingular projective curve,  
the Grothendieck-Lefschetz trace formula is written as the equality
$$\rC_o=H^0(\rC)-H^1(\rC)+H^2(\rC)=1-H^1(\rC)+\LL,$$
 where $\rC=\rC_o\times_{\overline{\FF_\pi}}\FF_\pi$, and $H^i(\rC)$ denotes the 
$i$-th \'etale cohomology group of $\rC$ with coefficients in $\Qlbar$, endowed with the action of Frobenius.
\item
Given a variety $\rC_o$ over $\FF_{\pi}$, we denote by $Z_H(\rC_o, q)$ its \emph{Hilbert zeta function}:
\begin{equation*}
Z_{H}(\rC_o, q):=\sum_{n=0}^\infty  \rC_o^{[n]}\cdot q^n \in K_0(\Var_{\FF_{\pi}})[[q]].
\end{equation*}
Note that this formal power series is invertible since it starts with $1$.
\item
Given a  $\Qlbar$-vector space $W$ with an action of Frobenius, we denote by  $\Lambda^*(-q W)$ the generating series of  its exterior powers:
\begin{equation*}\label{lambdaK}
\Lambda^*(-q W):= \sum_k (-q)^k \bigwedge^k  W\in K_0(\Rep(\Fr))[[q]].
\end{equation*}
This formal power series satisfies the identity 
\begin{equation}\label{lambdaK2}
\Lambda^*(-q (W_1 + W_2))= \Lambda^*(-q W_1)\Lambda^*(-q W_2).
\end{equation}
In particular, if $W_1, W_2$ are $\Qlbar$-vector spaces with trivial Frobenius action,
\begin{equation}\label{lambdaK3}
\Lambda^*(-q (W_1+W_2 \LL))=(1-q)^{\dim W_1}(1-q\LL)^{\dim W_2},
\end{equation}
a formula which we will often use.

Using this formalism,  the classical MacDonald formula \cite{mcd} for a nonsingular  (projective) curve $\rC_o$ with $r$ geometrically connected components which are defined over the base field  $\FF_{\pi}$ reads as: 
\begin{equation}\label{macdonald}
Z_H(\rC_o,q)=\frac{\Lambda^*(-q H^1(\rC) )}{\left((1-q)(1-q\LL) \right)^r}=\frac{\Lambda^*(-q H^1(\rC) )}
{\Lambda^*\left(-q\left(H^0(\rC)+H^2(\rC)\LL \right)\right) }.
\end{equation}

\item
Let  $\rC_o$ be a  nodal  curve  over $\FF_{\pi}$ and let 
$\Gamma=\Gamma_{\rC}$ be the dual graph of $\rC=\rC_o\times_{\FF_{\pi}} \ov{\FF_\pi}$.
For any $i=0,1$, we will set $H^i(\Gamma_{\rC}):=H^i(\Gamma_{\rC}, \Qlbar)$ and $H_i(\Gamma_{\rC}):=H_i(\Gamma_{\rC}, \Qlbar)$ endowed with the action of Frobenius (see \S \ref{S:dualgr}).

For any $I\subseteq \mathrm E(\Gamma_{\rC})$ and any $k\geq 0$, consider the map
$$\hat{N}_I^{(k)}:\bigwedge^k \left(H^1(\Gamma_\rC)+ H_1(\Gamma_\rC) \LL \right)\longrightarrow \bigwedge^k \left(H^1(\Gamma_\rC)+ H_1(\Gamma_\rC)  \LL \right) \LL^{|I|}
$$
defined in \S \ref{sec:CKS}. 
We now set 
\begin{equation*}
{\mathbf K}(\rC_o):=\sum_{k\geq 0} (-q)^k \sum_{I\subseteq \mathrm E(\Gamma_{\rC})} (-1)^{| I |} \Im \hat{N}_I^{(k)}  \in K_0(\Rep(\Fr))[[q]].
\end{equation*}
Using Lemma \ref{L:main}, it is easy to check that 
\begin{equation}\label{E:formK}
{\mathbf K}(\rC_o)=\sum_{I\in \mathscr{C}(\Gamma_{\rC})} (-q\LL)^{|I|} e_I^* \Lambda^*\left(-q\left(H^1(\Gamma_\rC\setminus I)+ H_1(\Gamma_\rC\setminus I) \LL \right)\right).
\end{equation}
Note that the homology and cohomology groups of $\Gamma_C\setminus I$  are not acted on  by the Frobenius unless the subset $I$ is Frobenius invariant. 
However, the sum on the right hand side of \eqref{E:formK} is the sum over all the subsets of $\mathscr{C}(\Gamma_{\rC})$, and   
is therefore acted on by the Frobenius, hence it belongs to $K_0(\Rep(\Fr))[[q]]$.
\end{enumerate}

\end{notation}

\begin{remark}\label{R:classCKS}
By the discussion in \S \ref{sec:CKS} (and using the Setup \ref{setup}), the class in $K_0(\Rep(\Fr))[[q]]$ of 
$$\sum_i q^i IC\left(\textstyle\bigwedge^i R^1 \pi_*{\Qlbar}_{|\Breg}\right)_b[-i]$$
is equal to 
$$\Lambda^*(-q H^1(\rC^\nu)) \cdot {\mathbf K}(\rC_o).$$
\end{remark}

Using the above notation, we can restate the main result of \cite{MY, MS2} as it follows.

\begin{theorem}(MacDonald formula for geometrical irreducible nodal curves, \cite{MY, MS2})\label{T:Mac-irr}
Let $\rC_o$ be a  {\em geometrically irreducible} nodal curve over $\FF_\pi$. 
Then the Hilbert zeta function of $\rC_o$ is equal to 
\begin{equation}\label{E:Mac-irr}
Z_H(\rC_o,q)=\frac{\Lambda^*(-q H^1(\rC^\nu))\cdot {\mathbf K}(\rC_o) }
{\Lambda^* \left( -q \left(H^0(\rC)+H^2(\rC)\LL \right) \right) } =\frac{\Lambda^*(-q H^1(\rC^\nu))\cdot {\mathbf K}(\rC_o) }
{(1-q)(1-q\LL)  }. 
\end{equation}
\end{theorem}


\medskip

The aim of this subsection is to generalize the above MacDonald formula to a (reducible) nodal curve $\rC_o$ defined over a finite field $\FF_\pi$, under the assumption that the geometrically irreducible components of $\rC_o$ are defined over the finite field $\FF_{\pi}$.



We will first find a formula for the Hilbert zeta function of $\rC_o$ in terms of the (co)homology of its spanning subgraphs.

\begin{proposition}\label{P:1for-Zeta}
Let $\rC_o$ be a nodal curve defined over a finite field $\FF_{\pi}$ of cardinality sufficiently big with respect to $\delta(\rC_o)$. 
Assume that the irreducible components of $\rC=\rC_o\times_{\FF_{\pi}} \ov{\FF}_{\pi}$ are defined over $\FF_{\pi}$.
Then we have the following formula for the Hilbert zeta function of $\rC_o$:
\begin{equation}\label{E:form1}
Z_H(\rC_o,q)=\Lambda^* \left(-qH^1(\rC^\nu)\right)
\left(
\sum_{J\subseteq \mathrm E(\Gamma_C)}(-q\LL)^{|J|}e_J^* \, 
\frac{\Lambda^* \left( -q (H^1(\Gamma_\rC \setminus J)+ H_1(\Gamma_\rC \setminus J) \LL \right)}{\Lambda^* \left(-q (H^0(\Gamma_\rC \setminus J) +H_0(\Gamma_\rC \setminus J) \LL  \right)}\right).
\end{equation}
\end{proposition}
As above, note that the homology and cohomology groups of $\Gamma_C\setminus J$  are not acted on  by the Frobenius unless the subset $J$ is Frobenius invariant. However, the sum on the right hand side of \eqref{E:form1} belongs to $K_0(\Rep(\Fr))[[q]]$, being the sum over all the subsets of $\mathrm E(\Gamma_{\rC})$. 


\begin{proof}
Since the cardinality of the finite field $\FF_{\pi}$ is big enough with respect to $\delta(\rC_o)$ (which is the number of nodes of $\rC$), we can find 
a rational curve $\rD_o$ with the same set of nodes of  $\rC_o$ and  of the same type (see the description of nodes in \S\ref{S:dualgr}).

Denoting $\nu: {\rC}_o^\nu \to \rC_o$ and $\nu : {\rD}_o^\nu \to \rD_o$
the normalization maps, by $\rC_{o, \mathrm{sm}}$ and $\rD_{o, \mathrm{sm}}$ the nonsingular sets, and by $\rC_{o, \times}=\rD_{o,\times}$ the singular (nodal) sets, we also have
\begin{equation*}\label{equa20}
\nu^{-1}(\rC_{o, \times})=\nu^{-1}(\rD_{o, \times}).
\end{equation*}

Recall how the Hilbert scheme of points factors into 
local contributions. 
Given a subset $\cS \subset C_o$, 
every point in $\rC_o^{[n]}$ is the union of a subscheme supported on $\cS$ and a subscheme
supported off $\cS$, whence a factorization  
\begin{equation*}\label{eq:descent}
Z_{H}(\rC_o,q)
= \left(\sum_{n=0}^\infty q^n \cdot (\rC_o \setminus \cS)^{[n]} \right) 
\cdot \left( \sum_{n=0}^\infty q^n  \cdot \rC_\cS^{[n]} \right)=
Z_{H}(\rC_o \setminus \cS,q)
\cdot  \left( \sum_{n=0}^\infty q^n  \cdot \rC_\cS^{[n]} \right),
\end{equation*}
where $\rC_\cS^{[n]}$ is the fibre of the Hilbert-Chow morphism over $\cS$.

Applying this to $\rC_o$ with $\cS= \rC_{o, \times}$, and to  $\rC_o ^\nu$ with $\cS=\nu^{-1}(\rC_{o, \times})$,
(resp. to $\rD_o$ with $\cS= \rD_{o, \times}$, and to  $\rD_o ^\nu$  with $\cS=\nu^{-1}(\rD_{o, \times})$), we find that
\begin{equation*}\label{equa30}
\frac{Z_{H}(\rC_o,q)}{Z_{H}(\rC_o ^\nu,q)}=
\frac{ Z_{H}(\rC_{o, \mathrm{sm}},q)(\sum (\rC_{o,\times})^{[n]}q^n)}{Z_{H}(\rC_{o, \mathrm{sm}},q)(\sum (\nu^{-1}(\rC_{o,\times}))^{[n]} q^n)}=\frac{ \sum (\rD_{o,\times})^{[n]}q^n}{\sum \nu^{-1}(\rD_{o,\times})^{[n]}q^n}
=\frac{\sum \rD_{o,\times}^{[n]}q^n}{\sum \nu^{-1}(\rD_{o,\times})^{[n]}q^n},
\end{equation*}
and similarly 
\begin{equation*}\label{equa31}
\frac{Z_{H}(\rD_o,q)}{Z_{H}(\rD_o^\nu,q)}=
\frac{ Z_{H}(\rD_{o, \mathrm{sm}},q)(\sum (\rD_{o,\times})^{[n]}q^n)}{Z_{H}(\rD_{o, \mathrm{sm}},q)(\sum (\nu^{-1}(\rD_{o,\times}))^{[n]} q^n)}
=\frac{\sum \rD_{o,\times}^{[n]}q^n}{\sum \nu^{-1}(\rD_{o,\times})^{[n]}q^n}.
\end{equation*}
Hence we conclude that
\begin{equation}\label{equa32}
\frac{Z_{H}(\rC_o,q)}{Z_{H}(\rC_o^\nu,q)}=\frac{Z_{H}(\rD_o,q)}{Z_{H}(\rD_o^\nu,q)}.
\end{equation}

Therefore, in order to conclude the proof, it remains to compute the Hilbert zeta functions of $\rC_o^\nu$, $\rD_o^\nu$ and $\rD_o$. 

The curve $\rD_o^{\nu}$ is smooth and rational, hence it is geometrically irreducible and with $H^1(\rD^{\nu})=0$. Hence MacDonald formula \eqref{macdonald}  for smooth curves gives that 
\begin{equation}\label{E:MacDnu}
Z_H(\rD_o^{\nu},q)=\frac{1}{(1-q)(1-q\LL) }.
\end{equation}

The curve $\rC_o^{\nu}$ is smooth and its geometrically irreducible components, whose number is equal to the cardinality $|V(\Gamma_{\rC})|$ of the dual graph $\Gamma_{\rC}$ of $\rC$, are defined over $\FF_{\pi}$ by our assumptions on $\rC_o$. Hence 
MacDonald formula \eqref{macdonald}  for smooth curves gives that 
\begin{equation}\label{E:MacCnu}
Z_H(\rC_o,q)=\frac{\Lambda^*(-q H^1(\rC) )}{\left((1-q)(1-q\LL) \right)^{|V(\Gamma_{\rC})|}}.
\end{equation}







The curve $\rD_o$ is geometrically irreducible; hence  MacDonald formula for geometrical irreducible nodal curves (see Theorem \ref{T:Mac-irr}) gives  that 
\begin{equation}\label{E:MacD}
Z_H(\rD_o,q)
=\frac{ {\mathbf K}(\rD_o)}{(1-q)(1-q\LL)}.
\end{equation}



We are left with computing ${\mathbf K}(\rD_o)$. Since  taking out any subset of edges does not disconnect $\Gamma_D$, we have that $\cC(\Gamma_{\rD})$ (see Definition \ref{D:defC}) is equal to the collection of all the subsets of the edge set $\mathrm E(\Gamma_{\rD})$ of $\Gamma_{\rD}$. Hence formula \eqref{E:formK} gives that 
\begin{equation}\label{E:KD}
{\mathbf K}(\rD_o)=\sum_{J\subseteq \mathrm E(\Gamma_{\rD})} (-q\LL)^{|J|}e_J^*\Lambda^*  \left(H^1(\Gamma_\rD \setminus J)
+ H_1(\Gamma_\rD \setminus J) \LL\right). 
\end{equation}

We now want to relate the (co)homology of the spanning subgraphs of $\Gamma_{\rC}$ with the ones of $\Gamma_{\rD}$. 
Observe that, by the construction of $\rD_o$ and the discussion in \S  \ref{S:dualgr}, the dual graphs  $\Gamma_{\rD}$ and $\Gamma_{\rC}$ have the same  set of oriented edges with the same Frobenius action, which implies that 
$\EE_{\Gamma_{\rD}}=\EE_{\Gamma_{\rC}}:=\EE$ and $E(\Gamma_{\rC})=\mathrm E(\Gamma_{\rD}):=\mathrm E$.
On the other hand, since $\rD_o$ is a rational curve, the vertex set of $\Gamma_{\rD}$ is one point with the trivial Frobenius action. Hence the exact sequences \eqref{E:seq} and \eqref{E:cohseq} applied to $\Gamma_{\rD}$ give that 
\begin{equation*} 
H_1(\Gamma_\rD)\cong \EE \: \text{ and } \: H^1(\Gamma_\rD)\cong \EE^*.
\end{equation*}
Substituting this into the  exact sequences \eqref{E:seq} and \eqref{E:cohseq} applied to $\Gamma_{\rC}$ and passing to the $K$-ring, we get the following equality in $K_0(\Rep(\Fr))$:  
\begin{equation*}\label{equagraf1}
 H_1(\Gamma_\rD) = H_1(\Gamma_\rC ) + \VV -  H_0(\Gamma_\rC  ) \: \text{ and } \:
 H^1(\Gamma_\rD ) = H^1(\Gamma_\rC ) + \VV^* -  H^0(\Gamma_\rC ),
 \end{equation*}
where   $\VV:=\VV_{\Gamma_{\rC}}$. Note that our assumption on the irreducible components of $\rC_o$ is equivalent to the fact that the action of Frobenius on the vertex set $V(\Gamma_{\rC})$ is trivial, hence the action of  Frobenius on  $\VV$  is trivial. 

The same relations hold between the graphs $\Gamma_\rC \setminus J$ and $\Gamma_\rD \setminus J$, obtained, respectively, from $\Gamma_\rC$ and $\Gamma_\rD$ by deleting  a set  $J\subset \mathrm E$  of edges, namely
\begin{equation*}\label{equagraf2}
 H_1(\Gamma_\rD \setminus J) = H_1(\Gamma_\rC \setminus J) + \VV -  H_0(\Gamma_\rC  \setminus J) \: \text{ and } \:
 H^1(\Gamma_\rD \setminus J) = H^1(\Gamma_\rC \setminus J) + \VV^* -  H^0(\Gamma_\rC  \setminus J).
 \end{equation*} 
Combining the above relations and using that $\VV\cong \VV^*$, we arrive at the relation
\begin{equation*}\label{equagraf3}
H^1(\Gamma_\rD \setminus J)+ H_1(\Gamma_\rD \setminus J)\LL=
H^1(\Gamma_\rC \setminus J)+ H_1(\Gamma_\rC \setminus J) \LL +\VV(1+\LL) -  H^0(\Gamma_\rC  \setminus J) - H_0(\Gamma_\rC  \setminus J)\LL .
\end{equation*}
By applying the operator $\Lambda^*(-q(-))$ to the above relation and using \eqref{lambdaK2} and \eqref{lambdaK3} (recall that the action of Frobenius on $\VV$ is trivial), we get 
\begin{equation}\label{E:LambdaCD}
\Lambda^*\left( H^1(\Gamma_\rD \setminus J)+ H_1(\Gamma_\rD \setminus J)\LL\right)=\Lambda^* \left(-q \left( \VV(1 + \LL ) \right) \right) \frac{\Lambda^*\left(H^1(\Gamma_\rC \setminus J)+ H_1(\Gamma_\rC \setminus J) \LL \right)}{\Lambda^*\left(H^0(\Gamma_\rC  \setminus J) + H_0(\Gamma_\rC  \setminus J)\LL \right)}=
\end{equation}
$$ 
= \left((1-q)(1-q\LL)\right)^{|\mathrm{V(\Gamma_{\rC})}|} \frac{\Lambda^*\left(H^1(\Gamma_\rC \setminus J)+ H_1(\Gamma_\rC \setminus J) \LL \right)}{\Lambda^*\left(H^0(\Gamma_\rC  \setminus J) + H_0(\Gamma_\rC  \setminus J)\LL \right)}. 
$$ 

Substituting \eqref{E:LambdaCD} into \eqref{E:KD}, we obtain 
\begin{equation}\label{E:KoD}
{\mathbf K}(\rD_o)=\left((1-q)(1-q\LL)\right)^{|\mathrm{V(\Gamma_{\rC})}|} \left( \sum_{J\subseteq \mathrm E}(-q\LL)^{|J|}e_J^* \, 
\frac{\Lambda^* \left(-q  \left( H^1(\Gamma_\rC \setminus J)+ H_1(\Gamma_\rC \setminus J) \LL \right)  \right)}
{\Lambda^* \left(-q  \left( H^0(\Gamma_\rC \setminus J)+ H_0(\Gamma_\rC \setminus J) \LL \right)  \right)} \right).
\end{equation}

We conclude by putting together \eqref{equa32}, \eqref{E:MacDnu}, \eqref{E:MacCnu}, \eqref{E:MacD} and \eqref{E:KoD}.  

\end{proof}


Now we want to express the right hand side of \eqref{E:form1} in terms of the operator ${\mathbf K}(-)$ applied to some special partial normalizations of the curve $\rC_o$, that we are now going to define.


Every subset $I$ of the edge set $ \mathrm E:=\mathrm E(\Gamma)$ of the dual graph $\Gamma:=\Gamma_\rC$ defines a partition $\lambda(I)$ of the vertex set $\mathrm V:=\mathrm V(\Gamma_{\rC})$: 
two vertices are in the same subset of the partition if they belong to the same connected component of the spanning subgraph 
$\Gamma \setminus I$.
The  partitions of $\mathrm V$ obtained this way will play a special role and we need a notation for them.

\begin{definition}\label{sm_partitions}
We denote by $\cP:=\cP(\Gamma)$ the set of partitions of the vertex set $\mathrm V$ of the form $\lambda(I)$, for some $I\subseteq \mathrm E$. 
\end{definition}

Given $\lambda \in \cP$, we let $S_\lambda$ to be the collection of subsets  $I \subseteq {\mathrm E}$ such that $\lambda(I)=\lambda$. 
Every   $S_\lambda$ has a minimal element $J_\lambda$ defined as follows: an edge belong to $J_\lambda$ if its end points belong to different subsets of the partition $\lambda$. We set $\delta(\lambda)=|J_\lambda|$.
Using the minimal element $J_{\lambda}$, we can give another description of $S_\lambda$: a subset $I\subseteq {\mathrm E}$ belongs to $S_\lambda$ if and only if $J_{\lambda}\subseteq I$ and the two graphs 
$\Gamma\setminus I $ and $ \Gamma\setminus J_\lambda$ have the same number of connected components. 

For any $\lambda\in \cP$,  set $\rC_\lambda$ be the (disconnected) nodal curve obtained from $\rC$ by normalizing the nodes in $J_\lambda$. Note that $\Gamma_{\rC_{\lambda}}=\Gamma_{\rC}\setminus J_{\lambda}$. 




The next theorem is the main result of this subsection. 
\begin{theorem}\label{vivek_nodal}
Same assumptions as in Proposition \ref{P:1for-Zeta}.
The  Hilbert zeta function of $\rC_o$ is equal to
\begin{equation}\label{E:form2}
Z_H(\rC_o,q)=
\Lambda^* \left(-qH^1(\rC^\nu)\right)\left(
\sum_{\lambda \in \cP(\Gamma_{\rC})} (q \LL)^{\delta(\lambda)}\frac{{\mathbf K }(\rC_\lambda)}
{\Lambda^* (-q \left( H^0(\Gamma_{\rC_\lambda})+ H_0(\Gamma_{\rC_\lambda}) \LL\right)} \right)
\end{equation}
\end{theorem}


\begin{proof}
Using Proposition \ref{P:1for-Zeta}, we have to show that the sum in the right hand side of \eqref{E:form1} is equal to the sum on the right hand side of \eqref{E:form2}.
 
 Note that we have a partition $\mathrm E(\Gamma)=\coprod_{\lambda\in \cP(\Gamma)}S_{\lambda}$, where $\Gamma:=\Gamma_{\rC}$. Moreover,  for each $J\in S_{\lambda}$, we have an inclusion of graphs $\Gamma\setminus J\subseteq \Gamma\setminus J_{\lambda}$ that induces a bijection on the number of connected components; hence we have that $H^0(\Gamma\setminus J)=H^0(\Gamma\setminus J_{\lambda})$ and $H_0(\Gamma\setminus J)=H_0( \Gamma\setminus J_{\lambda})$. 
  
The sum in the right hand side of \eqref{E:form1} can be written as
\begin{equation}\label{E:equ-sum}
\sum_{\lambda \in \cP(\Gamma)} \frac{ (q \LL)^{\delta(\lambda)}e^*_{J_\lambda}  }
 {\Lambda^*\left (-q \left( H^0(\Gamma \setminus J_\lambda)+ H_0(\Gamma \setminus J_\lambda) \LL\right)\right)}
\left(\sum_{J \in S_\lambda} (q \LL)^{|J|-\delta(\lambda)}e^*_{J \setminus J_\lambda}
\Lambda^*\left( -q\left(H^1(\Gamma \setminus J)+ H_1(\Gamma \setminus J) \LL \right) \right) \right).
\end{equation}
Remark that, since $J_\lambda$ is canonically attached to the partition $\lambda$, and since Frobenius acts trivially on this partition as we assumed that the geometric irreducible components of $\rC_o$ are defined over $\FF_{\pi}$, Frobenius acts trivially  on $e^*_{J_\lambda}$. This, together with the fact that $\Gamma_{\rC_{\lambda}}=\Gamma_{\rC}\setminus J_{\lambda}$, implies that we can rewrite \eqref{E:equ-sum} as
\begin{equation}\label{eqa_qf}
\sum_{\lambda \in \cP} \frac{ (q \LL)^{\delta(\lambda)}}
 {\Lambda^*\left (-q \left( H^0(\Gamma_{\rC_\lambda})+ H_0(\Gamma_{\rC_\lambda}) \LL\right)\right)}
\left(\sum_{J \in S_\lambda} (q \LL)^{|J|-\delta(\lambda)}e^*_{J \setminus J_\lambda}
\Lambda^*\left( -q\left(H^1(\Gamma_\rC \setminus J)+ H_1(\Gamma_\rC \setminus J) \LL \right) \right) \right)
\end{equation}
By the characterization of $S_{\lambda}$ given above, we  have that 
$$\mathscr{C}(\Gamma_{\rC_{\lambda}})=\{J\setminus J_{\lambda}\: : \: J\in S_{\lambda}\}.
$$
Moreover, for $J \in S_\lambda$ we have
$\Gamma_\rC \setminus J=\Gamma_{\rC_\lambda} \setminus (J\setminus J_\lambda)$ and  $|J|-\delta(\lambda)=|J\setminus J_\lambda|$.
Hence formula \eqref{E:formK} gives that 
\begin{equation}\label{E:K-Clambda}
\mathbf{K}(\rC_{\lambda})=\sum_{J \in S_\lambda} (q \LL)^{|J|-\delta(\lambda)}e^*_{J \setminus J_\lambda}
\Lambda^*\left( -q\left(H^1(\Gamma_\rC \setminus J)+ H_1(\Gamma_\rC \setminus J) \LL \right) \right). 
\end{equation}
Substituting \eqref{E:K-Clambda} into \eqref{eqa_qf}, we conclude that the sum in right hand side of \eqref{E:form1} is equal to the sum in the right hand side of \eqref{E:form2}, and this concludes the proof. 

\end{proof}

\section{Relative compactified Jacobian for non-versal families}\label{section:compjac_nvf}

The main result of this section, namely Theorem \ref{nonsing_rel_compjac}, gives sufficient conditions for a relative fine compactified Jacobian of a non-versal family to be nonsingular. In particular it allows the determination of the higher discriminants (see Definition \ref{high_discr_def}) for the relative compactified Jacobian of many families of planar curves. 
If a family of curves $\cC \to S$ contains only irreducible curves, then the relative compactified Jacobian is non singular if and only if the relative Hilbert schemes of any length are non singular \cite{S}. The if implication is still true for families of 
reducible curves (as we will show in Corollary \ref{hil_comp}), but the only if implication is no longer true:
already in arithmetic genus one, the "banana" curve, or a triangle of lines, give examples of fine compactified Jacobians which can be smoothed in a one-dimensional family, whereas the Hilbert scheme of length two of the curve needs at least a two-dimensional family. It should be clear from the proof of Theorem \ref{nonsing_rel_compjac} that the reason for this discrepancy is that certain torsion free sheaves, which, as points of the Hilbert scheme, can be smoothed only in a high dimensional family, cannot appear in the compactified Jacobian because of the stability condition. For instance, in the triangle, a torsion-free sheaf is contained in a fine compactified Jacobian if and only if it is locally free outside at most one point.

The proof of this fact, which we believe of independent interest, is based on the results of \cite{FGvS} and a local duality theorem due to T. Warmt \cite{W}, which we now review.
All the unproven facts here may be found in \cite[Chapter 4]{W} and \cite{FGvS}. 

Fix the following data:
\begin{enumerate}
\item 
a planar complete local ring $R= k[[x,y]]/(f)$, with $f=\prod_{a \in \Lambda} f_a$ and $f_a\in k[[x,y]]$ irreducible elements; assume that $k$ is an algebraically closed field of arbitrary characteristic.  
The set $\Lambda$ is the set of branches of $R$, i.e. minimal prime ideals of $R$, and we set $\lambda:= \sharp \Lambda$.
The normalization $\tilde{R}$ of $R$ is isomorphic to  $ \tilde{R}\simeq \prod_{a \in \Lambda} k[[T_a]]$, where $T_a$ is a parameter on the $a$-th branch. Observe that $\tilde{R} $ contains $R$ and it is a subring of the total fraction field 
 $Q(R)\simeq \prod_{a \in \Lambda} k((T_a))$.
 \item
a rank one, torsion-free $R$-module $M$, which, up to isomorphism, we can assume to contain $R$ and to be contained in $\wt{R}$:
$$R \subseteq M \subseteq \wt{R}.$$
\end{enumerate}
Consider the conductor ideal of the extension $\tilde{R}/R$ 
\begin{equation*}
\mathfrak{f}=
\mathrm{Ann}( \widetilde{R} / R)=\Hom_R(\widetilde{R}, R)=\{ u \in  R \mbox{ such that } u \widetilde{R} \subset R \},
\end{equation*}
which is the biggest ideal of $\tilde{R}$ contained in $R$. The delta-invariant of the ring $R$ is defined as  $\delta(R):=\dim \wt{R}/R$.
Since $R$ is Gorenstein by our assumptions, we have that 
\begin{equation}\label{gor}
\delta(R)=\dim {R}/\mathfrak{f}=\frac{1}{2} \dim \wt{R}/\mathfrak{f}.
\end{equation}

One can associate to the module $M$ two objects of primary importance:

\begin{itemize}

\item 
The first Fitting ideal $\mathrm{Fit}_1(M)$ of $M$, defined as the ideal generated by $(N-1)$-minors of a free resolution
$$
0 \longleftarrow M  \longleftarrow k[[x,y]]^N \longleftarrow k[[x,y]]^N  \longleftarrow 0
$$
of $M$ as a $k[[x,y]]$-module.
Under the hypotheses above, we have that 
\begin{equation*}
\mathrm{Fit}_1(M)= \{ \phi(m), \mbox{ for } m\in M \mbox{ and } \phi \in \mathrm{Hom}_R (M, R) \},
\end{equation*}
and  $\mathfrak{f} \subseteq \mathrm{Fit}_1(M)$, see \cite[Prop. C-2 and Cor. C-3]{FGvS}.

\item
The endomorphism ring of $M$
\begin{equation*}
\mathrm{End}_R(M)=\{ c \in \wt{R}\: : \: cm \in M \mbox{ for all } m \in M\}.
\end{equation*}
which is a subring of $\wt{R}$ containing $R$ and contained in $M$. 
Notice that $\mathrm{End}_R(M)$ may not be planar and not even Gorenstein.

\end{itemize}

We have the series of inclusions
\begin{equation*}
\mathfrak{f} \subseteq   \mathrm{Fit}_1(M)   \subseteq R \subseteq \mathrm{End}_R(M) \subseteq M \subseteq \wt{R}.
\end{equation*}

The  first Fitting ideal of $M$ is dual to the endomorphism ring of  $M$, as stated in the following result. 

\begin{proposition}\label{kkl}(\cite[Korollar 4.4.2, ii]{W})
Under the hypotheses above, the map 
\begin{equation*}
\begin{aligned}
 \mathrm{Hom}_R( \mathrm{End}_R(M) , R) & \longrightarrow \mathrm{Fit}_1(M) \\
 \Psi & \mapsto \Psi(\id) \\
 \end{aligned}
\end{equation*}
is an isomorphism.
\end{proposition}

Using the endomorphism ring of a module $M$, we can introduce an important numerical invariant of $M$.

\begin{definition}\label{def:typeM}
Let $\nu=(\lambda_1, \ldots, \lambda_{l(\nu)})$ be a partition of $\lambda= \sharp \Lambda$.
We say that $M$ has type $\nu$ if $\mathrm{End}_R(M)$ is direct product of $l(\nu)$ local rings, 
the $i$-th of which has $\lambda_i$ branches. The type of $M$ is denoted by $\nu(M)$. 
\end{definition}

Given a partition $\nu=(\lambda_1, \ldots, \lambda_{l(\nu)})$ as above, let 
$$
I_1=\{1, \cdots, \lambda_1\}, I_2=\{\lambda_1+1, \cdots , \lambda_1+\lambda_2\}, \cdots I_{l(\nu)}=\{\lambda-\lambda_{l(\nu)}+1, \cdots , \lambda\}$$ 
and let $R_\nu$ 
be the subring of $\tilde{R}$ given by 
$$R_\nu=\{(f_1(T_1), \cdots , f_{\lambda}(T_{\lambda})) \in  \prod_{i=1}^\lambda k[[T_i]] \mbox{ with } f_k(0)=f_l(0) \mbox{ if } k,l \in I_j \mbox{ for some }j \}.$$
Geometrically, $R_{\nu}$ is the disjoint union of the complete local rings at $0$ of the coordinate axes in $\mathbb{A}^{\lambda_i}$, for $i=1, \ldots, l(\nu)$. Therefore, the rings $R_{\nu}$ are seminormal and, indeed, they are all
the seminormal rings containing $R$ and contained in $\widetilde{R}$. 
If a partition $\nu'$ refines $\nu$ then we have that $R_{\nu}\subseteq R_{\nu'}$; the two extreme case being $R_{(1,\ldots, 1)}=\tilde{R}$ and $R_{(\lambda,0,\ldots,0)}$ which is the seminormalization of $R$. 
The delta invariant of $R_{\nu}$ is easily seen to be equal to 
\begin{equation}\label{deltaRnu}
\dim \tilde{R}/R_\nu:=\delta(R_\nu)=\lambda-l(\nu).
\end{equation}
From Proposition \ref{kkl} and using that $\mathfrak{f}=\Fit_1(\widetilde{R})$, we deduce that 
\begin{equation}\label{FitRnu} 
\dim \Fit_1(R_{\nu})/\mathfrak{f}=\lambda-l(\nu).
\end{equation}

From \cite[Chapter 7, Ex. 5.9]{Liu}, we deduce the following alternative characterization of the type of $M$.
\begin{lemma}\label{typeM}
The type of $M$ is the coarsest partition $\nu$ such that $\mathrm{End}_R(M)\subseteq R_\nu$. Hence, $R_{\nu(M)}$ is the seminormalization of $\End_R(M)$. 
\end{lemma}

From the above characterization of the type of $M$ and Proposition \ref{kkl}, we deduce the following:

\begin{corollary}\label{FittM}
For any module as above, we have that 
$\mathrm{Fit}_1(M) \supseteq  \mathrm{Fit}_1(R_{\nu(M)})$.
\end{corollary}

We now review the nonsingularity condition for a relative fine compactified Jacobian at a given point: the reference is again \cite{FGvS}.
A clear recollection of the results can be found in \cite[\S 4.5]{W}.

Let $\rC$ be a projective reduced connected curve with planar singularities over $k=\ov{k}$ , $\rC_{\mathrm{sing}} =\{c_1 , \cdots, c_r\}$ its singular set,
$\{\Lambda_1 , \cdots, \Lambda_r\}$ the corresponding sets of branches, with cardinality $\lambda_i:=\sharp \Lambda_i$.

Given a singular point $c_i \in \rC_{\mathrm{sing}}$, let $f_i$ be a local equation of $C$ at $c_i$, so that $\hat{\cO}_{\rC, c_i} \simeq k[[x,y]]/(f_i).$ 
We have the deformation functor $\VV_{i}:=\Def_{\rC, c_i}$ of the local ring $\hat{\cO}_{\rC, c_i}$, 
whose tangent space $T\VV_i$ is the underlying vector space of the $k$-algebra $k[[x,y]]/(f_i, \partial_x f_i,  \partial_y f_i).$
There is the canonical subspace $\VV_{i}^\delta \subset T \VV_{i}$, the support of the tangent cone at $c_i$ of the equigeneric locus.
The subspace $\VV_{i}^\delta$ is the class in $T\VV_i=k[[x,y]]/(f_i, \partial_x f_i,  \partial_y f_i)$ of the conductor ideal $\mathfrak{f}_{i} :=\mathrm{Ann}( \widetilde{\mathcal O}_{\rC, c_i} / \cO_{\rC, c_i})$.
By (\ref{gor}),  we have that 
\begin{equation}\label{codeqgen}
\codim \VV_{i}^\delta=\delta (c_i) = \dim_k \widetilde{\mathcal O}_{\rC, c_i} / \cO_{\rC, c_i}.
\end{equation}
Given a partition $\nu_i$ of the set $\Lambda_i$ of branches at $c_i$, we have the
partial normalization with local ring $({{\mathcal O}_{\rC, c_i}})_{\nu_i}$ and the  subspace $\VV_i^{\nu_i}$, representing the class in $T\VV_i=k[[x,y]]/(f_i, \partial_x f_i,  \partial_y f_i)$ of the ideal $\mathrm{Fit}_1(({{\mathcal O}_{\rC, c_i}})_{\nu_i})$. 
By  \eqref{FitRnu},  we have that 
\begin{equation}\label{esti1}
\dim \VV_i^{\nu_i}/ \VV_i^{\delta}=\dim \mathrm{Fit}_1(({{\mathcal O}_{\rC, c_i}})_{\nu_i})/\mathfrak{f}_i=\lambda_i - l(\nu_i).
\end{equation} 

We set  $\VV:=\Def^{\rm loc}_\rC=\prod \VV_i$  and  $\VV^\delta:=\prod \VV_i^\delta\subset T\VV=\prod T\VV_i$, a codimension $\delta(\rC)=\sum \delta(c_i)$ linear subspace.
Given a multipartition  $\underline{\nu}=\{ \nu_i \}$, where $\nu_i$ is a partition of $\lambda_i$, we have the subspace $\VV^{\underline{\nu}}:=\prod \VV_i^{\nu_i}\subset T\VV$
and the corresponding partial normalization $\rC^{\underline{\nu}}$ of $\rC$, with local ring $(\cO_{\rC,c_i})_{\nu_i}$ at the point $c_i \in \rC$. The curve $C^{\un{\nu}}$ is seminormal and indeed all seminormal partial normalizations of $\rC$ 
are of the form $\rC^{\un{\nu}}$ for some unique multipartition $\un{\nu}=\{\nu_i\}$.  By \eqref{esti1}, we get that �
\begin{equation}\label{cod1}
 \codim \VV^{\underline{\nu}}=\sum_{i=1}^r \codim \VV^{\nu_i}_i = \sum_{i=1}^r (\delta(c_i)+ l(\nu_i)   - \lambda_i)=\delta(\rC) + \sum_{i=1}^r ( l(\nu_i)   - \lambda_i).
 \end{equation}  

Let $\I$ be a rank one torsion free sheaf  on $\rC$  with  stalk $\I_i$ at $c_i$.
The deformation functor  $\mathrm{Def}((\rC, c_i),\I_i)$ of the pair $(\hat{\cO}_{\rC, c_i}, \I_i)$ is endowed with a forgetful morphism  
$\rho_i: \mathrm{Def}((\rC, c_i),\I_i) \to \VV_i=\mathrm{Def}_{\rC, c_i}$ and we set 
$$
\rho:=\prod \rho_i : \prod \mathrm{Def}((\rC, c_i),\I_i) \to \VV=\Def^{\rm loc}_\rC.
$$ 
Let  $W_i(\I)=\mathrm{Im} (d \rho _i)$ and $W(\I) =\mathrm{Im}(d \rho)= \prod W_i(\I)$ be the images of the differentials.

The linear subspace $W_i(\I)$ is determined by the first Fitting ideal of $\I_i$.  
   
\begin{proposition} \label{W_i} (\cite[Prop. C-1]{FGvS})
The subspace $W_i(\I)$ is the class in $k[[x,y]]/(f_i, \partial_x f_i,  \partial_y f_i)$ of the first Fitting ideal $\mathrm{Fit}_1(\I_i)$ of the stalk $\I_i$ of $\I$ at $c_i$.
\end{proposition} 

The linear subspace $W(\I)$  allows to characterize when a relative fine compactified Jacobian is regular at the point $\I$.   Recall that  given   a family of curves   $\pi: \cC \to B$ and a point $b\in B$ such that $\rC:=\pi^{-1}(b)=\cC_{b}$,  we have the local Kodaira-Spencer map (see \eqref{E:KSloc}):
$$k^{\rm loc}_{\pi,b}:T_b(B) \to T\Def^{\rm loc}_{\rC}=T\VV.$$


\begin{proposition}\label{nonsing_crit}(\cite[Cor. B-3]{FGvS})
Given a family $\pi: \cC \to B$, with $\rC=\cC_b$, a relative fine compactified Jacobian $\ov{J}_{\cC}$   is regular at a  point $\I$ lying in the central fiber $(\ov{J}_{\cC})_o=\ov J_{\rC}$ if and only if  $W(\I) + \Im(k^{\rm loc}_{\pi, b}) =T \VV$.
\end{proposition}

Consider  the endomorphism sheaf $\underline{\mathrm{End}}_{\oO_{\rC}}(\I)$ of $\I$: it is a sheaf of finite $\oO_{\rC}$-algebras such that
$\oO_{\rC} \subseteq \underline{\mathrm{End}}_{\oO_{\rC}}(\I) \subseteq \oO_{{\rC^\nu}}$. The sheaf $\I$ is naturally a sheaf on the partial normalization $ \rC^{\I}:= \underline{\mathrm{Spec}}_{\rC}( \mathrm{End}_{\oO_{\rC}}(\I) )$ of $C$; the original $\I$ being
recovered by the pushforward along the partial normalization morphism $\nu_{\I}: \rC^{\I} \to \rC$. 
For every singular point $c_i$ of $\rC$, denote by $\nu_i(\I_i)$ the type of $\I_i$ at $c_i$ (see Definition \ref{def:typeM}) and we set $\underline{\nu}(\I)=\{\nu_i(\I_i)\}$. It follows from Lemma \ref{typeM} that $\rC^{\underline{\nu}(\I)}$ is the seminormalization of $\rC^{\I}$. 
The following remark is obvious and it is recorded for later use.

\begin{remark}\label{conH}
The sheaf $\I$ is simple if and only if $\rC^{\I}$ is connected, or equivalently, if and only $C^{\un{\nu}(\I)}$ is connected. In particular, if $\I$ belongs to some fine compactified Jacobian of $\rC$, then  $C^{\un{\nu}(\I)}$ are connected.
\end{remark}

We want now to establish a necessary combinatorial criterion in order to check when the partial normalization $C^{\un{\nu}}$ is connected. 

To any  reduced projective curve $\rC$ (not necessarily locally planar), we associate an \emph{hypergraph} $H_{\rC}=(V(H_{\rC}), E(H_{\rm C}))$ as follows: the vertices $V(H_C)$ correspond to the irreducible components of 
$\rC$ and to each singular point $n\in \rC_{\rm sing}$ we associate an hyperedge $e_n$  which is a multiset of $V(H_C)$ consisting of all irreducible components that contain $n$, 
each one of which counted  with multiplicity equal to its number of branches at $n$. In this way, the cardinality $|e_n|$ of the hyperedge $e_n$ is equal to the total number of branches of $\rC$ at $n$.  
Note that if $\rC$ is a nodal curve, then the hypergraph $H_{\rC}$ is actually a graph and it coincides with the dual graph of $\rC$.  

\begin{lemma}\label{conhyp}
If the curve $\rC$ is connected then 
$$b(H_C):=\sum_{e\in E(H_C)} (|e|-1)-|V(H_{\rC})|+1\geq 0.$$ 
\end{lemma}
\begin{proof}
Clearly the curve $\rC$ is connected if and only if its associated hypergraph $H_{\rC}$ is connected, i.e. there does not exist a partition of the vertex set $V(H_C)=V_1\coprod V_2$ such that every hyperedge $e$ contains only elements of
either  $V_1$  or $V_2$. We will therefore prove more generally that if a hypergraph $H=(V(H),E(H))$ is connected then $b(H)\geq 0$. 

In order to show this, consider the bipartite simple incidence graph $\Gamma_H$ constructed from $H$ as follows: its vertices $V(\Gamma_H)$  are the disjoint union of $V(H)$ and of $E(H)$ and its edges 
are given by $E(\Gamma_H):=\{(v,e)\in V(H)\coprod E(H) : \: v\in e\}$.  Clearly $H$ is connected if and only if $\Gamma_H$ is connected and, by construction, 
 we have that $|V(\Gamma_H)|=|V(H)|+|E(H)|$ and $|E(\Gamma_H)|=\sum_{e\in E(H)} |e|$.  Therefore, if $H$ is connected then $b(H)$ coincides with the first Betti number of $b_1(\Gamma_H)=|E(\Gamma_H)|-|V(\Gamma_H)|+1$ of $\Gamma_H$, which is non-negative.
\end{proof}

We are now ready to prove the main result of this section, namely a sufficient criterion for the regularity of relative fine compactified Jacobians.  The criterion will be expressed in terms of  the following closed subset of $T\VV$:

\begin{definition}\label{defW}
Let $\rC$ be a curve as above. Consider the closed locus $\WW\subset T\VV$ given by the union of the linear subspaces $\VV^{\un{\nu}}$, as $\un{\nu}$ varies among all  the maximal multipartitions such that $C^{\un{\nu}}$ is connected.
\end{definition}

The  locus $\WW$ has the following properties:

\begin{lemma}\label{propW}
\noindent 
\begin{enumerate}[(i)]
\item \label{propW1} The locus $\WW\subset T\VV$ has pure codimension $\delta^a(\rC)$.
\item \label{propW1bis} We have the inclusion $\VV^{\delta}\subseteq \WW$ with equality if and only if $\rC$ is irreducible.  
\item \label{propW2} If $\rC^{\un{\nu}}$ is connected then $\VV^{\un{\nu}}$ contains some irreducible component of $\WW$. 
\item \label{propW3} If $\I$ is a simple torsion-free rank one sheaf then $W(\I)$ contains some irreducible component of $\WW$. 
\end{enumerate}
\end{lemma}
\begin{proof}
Part \eqref{propW1}: the irreducible components of $\WW$ are given by $\VV^{\un{\nu}}$, where $\un{\nu}$ is a maximal multipartition such that $\rC^{\un{\nu}}$ is connected. Lemma \ref{conhyp} implies that �$b(H_{\rC^{\un{\nu}}})\geq 0$. 
However, due to the maximality of $\un{\nu}$ we must have that $b(H_{\rC^{\un{\nu}}})=0$, for otherwise it is easy to check that we could find a refinement $\un{\nu'}$ with $\rC^{\un{\nu'}}$ still connected, violating the maximality of $\un{\nu}$. 
(This argument is the analogue for a hypergraph of the fact that every connected graph has a spanning tree.)
By the definition of $\rC^{\un{\nu}}$, it follows that if $\un{\nu}=\{\nu_i\}$ with $\nu_i=((\nu_i)_1,\ldots,(\nu_i)_{l(\nu_i)})$ a partition of the set $\Lambda_i$  of branches of $\rC$ at $c_i$, then the hyperedges of $H_{\rC^{\un{\nu}}}$ have cardinality 
$(\nu_i)_j$. Therefore, by the definition of $b(H_{\rC^{\un{\nu}}})$, we get
 \begin{equation}\label{bH}
 0=b(H_{\rC^{\un{\nu}}})=\sum_{i=1}^r \sum_{k=1}^{l(\nu_i)}((\nu_i)_k-1)-|V(H_{\rC^{\un{\nu}}})|+1=\sum_{i=1}^r(\lambda_i-l(\nu_i))-\gamma(\rC)+1.
 \end{equation}
 Combining \eqref{cod1} and \eqref{bH}, we deduce that 
 $$\codim \VV^{\un{\nu}}=\delta(\rC)-\gamma(\rC)+1=\delta^a(\rC), $$ 
  which concludes the proof of part \eqref{propW1}.
  
 Part \eqref{propW1bis}: consider the maximal multipartion  $\un{\nu}_{\rm max}$, i.e. the one for which each partition $\nu_i$ appearing in it has the form $\nu_i=(1,\ldots,1)$. From the above discussion, it follows that 
 $\rC^{\un{\nu}_{\rm max}}$ is the normalization $\widetilde{\rC}$ of $\rC$ and that $\VV^{\delta}=\VV^{\un{\nu}_{\rm max}}$. Therefore, we have the inclusion $\VV^{\delta}\subseteq \WW$ by Proposition \ref{kkl} and Proposition \ref{W_i}, and equality
 holds if and only if $\widetilde{\rC}$ is connected, which holds if and only if $\rC$ is irreducible.  �

Part \eqref{propW2}:  if $\un{\nu}$ is a multipartition such that $\rC^{\un{\nu}}$ is connected, then as observed above we can find a refinement $\un{\nu'}$ of $\un{\nu}$ such that $C^{\un{\nu'}}$ is connected and it is maximal with this property. 
 Therefore $\VV^{\un{\nu'}}$ is an irreducible component of $\WW$ and $\VV^{\un{\nu}}\subseteq \VV^{\un{\nu'}}$ by Proposition \ref{kkl}, q.e.d. �
  
Part \eqref{propW3}: if $\I$ is a simple torsion-free rank one sheaf then the seminormalization $\rC^{\un{\nu}(\I)}$ of $\rC^{\I}$ is connected (see Remark \ref{conH}) and we have that $\VV^{\un{\nu}(\I)}\subseteq W(\I)$ by Corollary \ref{FittM} and 
Proposition \ref{W_i}.  Therefore, we conclude by part \eqref{propW2}.
\end{proof}

Finally we can state and prove the main result of this section.

\begin{theorem}\label{nonsing_rel_compjac}
Let $\pi: \cC \to S$ be a projective  family  of connected curves, with $\rC=\cC_b$ having locally planar singularities, and let 
$k^{\rm loc}_{\pi,b}:T_b(B) \to T\Def^{\rm loc}_{\rC}=T\VV$
be the local Kodaira-Spencer map (see \eqref{E:KSloc}).
Let $\WW \subseteq T \VV$ be the locus of Definition \ref{defW}. Then a relative fine compactified Jacobian $\ov{J}_{\cC}$  is regular along $(\ov{J}_{\cC})_o=\ov{J}_{\rC}$ if 
$\Im(k^{\rm loc}_{\pi, b})$ is transverse to each irreducible component of $\WW$. In particular, this is the case if $\Im(k^{\rm loc}_{\pi, b})$ is a generic subspace of $T \VV$ of dimension at least $\delta^a(\rC)$.
\end{theorem}
\begin{proof} 
By Proposition \ref{nonsing_crit}, a relative fine compactified Jacobian $\ov{J}_{\cC}$  is regular along $(\ov{J}_{\cC})_o=\ov{J}_{\rC}$ if and only if $\Im(k^{\rm loc}_{\pi, b})$ is transverse to any linear subspace $W(\I)$ for any sheaf $\I\in \ov{J}_{\rC}$.
By Remark \ref{conH} and Lemma \ref{propW}\eqref{propW3}, any such linear subspace $W(\I)$ contains an irreducible component of $\WW$; therefore, if  $\Im(k^{\rm loc}_{\pi, b})$ is transverse to each irreducible component of $\WW$, then $\Im(k^{\rm loc}_{\pi, b})$ is transverse to every such linear subspace $W(\I)$ and the regularity of  $\ov{J}_{\cC}$  along $(\ov{J}_{\cC})_o=\ov{J}_{\rC}$ follows. 

Since $\WW$ has pure codimension $\delta^a(\rC)$ by Lemma \ref{propW}\eqref{propW1}, a generic linear subspace of dimension $\delta^a(\rC)$ is transverse to every irreducible component of $\WW$. 
\end{proof}

\begin{example}
Let $\rC$ be the banana curve.  Then $\Def(\rC)$ is 2-dimensional, since $\rC$ has two nodes.
We have $\delta^a(\rC) = \delta(\rC) + 1 - \gamma(\rC) = 2 + 1 - 2 = 1$, and indeed the relative
fine compactified Jacobian of a general 1-parameter family containing a banana curve is smooth
-- indeed, it is the family itself. 
\end{example}

\begin{example}\label{hitchin}
Let $\rC$ be a nonsingular projective curve of genus $g\geq 2$.
Let $h: \cM \to \AA$ be the Hitchin fibration 
for Higgs bundles over $\rC$ of rank $n$ and degree $d$ with $(d,n)=1$. 
We have the spectral curve family  $\pi: \cC \to \AA$:
For every $\underline{a} \in \AA$ the fibre $h^{-1}(\underline{a})$ is isomorphic to the 
fine compactified Jacobian of the spectral curve $\cC_{\underline{a}}=\pi^{-1}(\underline{a})$, mapping $n:1$ to $\rC$. 
Reducible spectral curves consist of a union of curves $\cC_i$ mapping $n_i :1$ to $\rC$, with $\sum n_i=n$.
For such a curve, the polarization of the corresponding Jacobian is described in Appendix A in \cite{MRV2}.
In this case, the loci where $\delta^{a}=r$ have exactly codim $r$, that is, by Theorem \ref{nonsing_rel_compjac}, the Hitchin system exhibits the 
minimal transversality to the $\delta^a$ loci  which is allowed in order to have a smooth total space.  
\end{example}

\begin{remark}\label{sharp}
The regularity criterion in Theorem \ref{nonsing_rel_compjac}  is sharp (in other words, the only if implication is also true) if the following Conjecture is true:

\begin{conjecture}
 Let $\ov J_{\rC}$ be a fine compactified Jacobian of a (reduced and projective) connected curve $\rC$ with planar singularities and let $\rC^{\un{\nu}}$ be a connected seminormal partial  normalization of $\rC$ that is maximal 
with these properties (or, even more generally, any connected partial normalization of $\rC$). Then there exists a sheaf $\I\in \ov J_{\rC}$ such that $\rC^{\I}=\rC^{\un{\nu}}$. 
\end{conjecture}
\vspace{0.1cm}

The above Conjecture is easily checked to hold if $\rC$ is irreducible: in this  case 
the unique $\rC^{\un{\nu}}$ as in the statement of the conjecture is the normalization $\widetilde{\rC}$ of $\rC$ and it is enough to take $\I=\nu_*(L)$ for a line bundle $L$ on $\widetilde{\rC}$ of suitable degree. Therefore,  if $\rC$ is irreducible 
we have that $\WW=\VV^{\delta}$ by Lemma \ref{propW}\eqref{propW1bis} and Theorem \ref{nonsing_rel_compjac}� above is sharp.

The above Conjecture holds true for nodal curves by \cite[Thm. 5.1]{MV}; in particular, Theorem \ref{nonsing_rel_compjac} �is sharp if $\rC$ is a nodal curve. 
\end{remark}

Finally we compare the nonsingularity of relative fine compactified Jacobians \ref{nonsing_rel_compjac} with that of the relative Hilbert schemes.
As a consequence of the results in  \cite{S} the following holds:

  \begin{theorem} \label{thm:smoothness_hilb}
  Let $\pi: \cC \to S$ be a projective  family  of  (non necessarily connected) curves, with $\rC=\cC_b$ having locally planar singularities, let  $k^{\rm loc}_{\pi,b}:T_b(B) \to T\Def^{\rm loc}_{\rC}=T\VV$
be the local Kodaira-Spencer map (see \eqref{E:KSloc}) and let  $\cC^{[d]}\to B$ be the relative Hilbert scheme of length $d$. Then: 
    \begin{enumerate}
    \item The regularity of $\cC^{[d]}$ along  $(\cC^{[d]})_o=\rC^{[d]}$ depends only on $\mathrm{Im}(k^{\rm loc}_{\pi, b})$.
    \item If $\cC^{[d]}$ is regular along $\rC^{[d]}$, then $\dim \mathrm{Im}(k^{\rm loc}_{\pi, b}) \ge \min(d, \delta(\rC))$. 
    \item   $\cC^{[d]}$ is regular along $\rC^{[d]}$ for all $d$ if and only if
      $\mathrm{Im}(k^{\rm loc}_{\pi, b})$ is transverse to  $\VV^{\delta}$.  In particular, this is the case if $\Im(k^{\rm loc}_{\pi, b})$ is a generic subspace of $\VV$ of dimension at least $\delta(\rC)$.
    \end{enumerate}
  \end{theorem}

\begin{corollary}\label{hil_comp}
Let $\pi:\cC\to S$ be as in Theorem \ref{nonsing_rel_compjac}. If $\cC^{[d]}$ is regular along $(\cC^{[d]})_o=\rC^{[d]}$ for all $d$, then 
any  relative fine compactified Jacobian $\ov{J}_{\cC}$  is regular along $(\ov{J}_{\cC})_o=\ov{J}_{\rC}$.
\end{corollary}
\begin{proof}
It follows by comparing Theorem \ref{nonsing_rel_compjac} with Theorem \ref{thm:smoothness_hilb} and using that $\VV^{\delta}\subseteq \WW$. 
\end{proof}

The implication in the above Corollary can be reversed if $\rC$ is irreducible because in this case $\WW=\VV^{\delta}$ by Lemma \ref{propW}\eqref{propW1bis};
if  $\rC$ is a reducible nodal curve
this is not true, in view  of Remark \ref{sharp}, and  we expect it not to be true for every reducible curve with planar singularities. This would follow from 
Conjecture 5.14.

\section{Support theorems for versal families}\label{supp_thm_vf}

In this section, relying on the results of \S \ref{nodal_curves_section} and \S \ref{section:compjac_nvf}, we establish Theorems \ref{thm:jac_vf} and \ref{thm:hilb}, which are the main
results of this paper. In this section we work over an algebraically closed field. 
  
The results of \S \ref{section:compjac_nvf} can be interpreted as determining the higher discriminants of the relative compactified jacobian and relative Hilbert scheme families.
This allows us to reduce the determination of the supports to the nodal locus, which is precisely what we did in \S \ref{nodal_curves_section}.

\subsection{Higher discriminants}\label{high_discr}

Higher discriminants \cite{MS} give a-priori bound on supports which may appear
in the direct image of the constant sheaf by a proper map.

\begin{definition}\label{high_discr_def}
Let $f: X \to Y$ be a proper map between nonsingular varieties. For any $i\geq 1$, the $i$-th discriminant 
$\Delta^i (f)$ is  the  locus of $y \in Y$ such that
there is no $(i-1)$ dimensional subspace of $T_y Y$ transverse to $df_x(T_x X)$ for every $x \in f^{-1}(y)$.
\end{definition}
Observe that the $i$-th discriminants $\Delta^i (f)$ form a chain of closed subsets and moreover $\Delta^1(f)$ is the discriminant locus of the map $f$, i.e. the complement of the biggest open subset of $Y$ where the restriction of the morphism $f$ is a smooth morphism.

\begin{theorem} \cite[Theorem 3.3]{MS2}\label{hd_general}
Let $f:X \to Y$ be a projective map between 
algebraic varieties, with $X$ nonsingular.   
Let $\cG$ be a summand of  $Rf_* \QQ_\ell$, and let  $k$ be the codimension of $\mathrm{supp}\,\cG$.
Then 
$$\mathrm{supp}\, \cG \subseteq \Delta^k(f).$$

In particular, if, for every $k$, we have that 
\begin{equation}\label{codest}
\codim \Delta^k (f)\ge k \mbox{ for all } k, 
\end{equation}
 then every summand
of $Rf_* \QQ_\ell$ is supported on the closure of a $k$-codimensional component of $\Delta^k (f)$.
\end{theorem}

Notice that, over the complex numbers, it follows easily from the existence of stratifications that the estimate \ref{codest} always holds.  
The following Theorem, an easy consequence of the results of \S \ref{section:compjac_nvf}, gives a description of 
the higher discriminants loci of the relative Hilbert scheme and of any relative fine compactified Jacobian for a versal family in terms of $\delta$ 
(resp. $\delta^a$)-loci. As a consequence, estimate \ref{codest} holds over any algebraically closed field for the maps $\pi^J$ and $\pi^{[n]}$.
\begin{theorem}\label{hd_families}
Let $\pi:\C \to B$ be a projective versal family of curves with locally planar singularities, let  $\pi^J:\ov{J}_{\cC}\to B$ 
be a relative fine compactified Jacobian (which exists after passing to an \'etale cover of $B$ by Theorem \ref{T:fam-Jac}),
and let $\pi^{[n]}: \cC^{[n]} \longrightarrow B$ be the relative Hilbert scheme  of length $n$.

Then we have:
\begin{enumerate}
\item \label{discrcompjac} 
The $i$-th discriminant of $\pi^J$ is equal to 
\begin{equation}\label{hd_cj}
\Delta^i(\pi^J)=\{b \in B \text{ such that } \delta^a(\cC_b) \geq i \}.
\end{equation}
Moreover,  the geometric generic point of each  codimension $i$ irreducible component of $\Delta^i(\pi^J)$ is an  irreducible nodal curve. 
\item\label{dischilb}
For every $n$, we have
\begin{equation}\label{hd_hs}
\Delta^i(\pi^{[n]}) \subseteq \{b \in B \text{ such that } \delta(\cC_b)\geq i \}.
\end{equation}
Moreover,  the geometric generic point of each   irreducible component  of $\Delta^i(\pi^J)$ and of $\Delta^i(\pi^{[n]})$ is a nodal curve.
\end{enumerate}
\end{theorem}

\begin{proof}
Statement \eqref{discrcompjac}: the first part  follows from Theorem \ref{nonsing_rel_compjac}. 
For the second part: if $\cC_{\ov \eta}$ is a geometric generic point of a component of codimension $i$ of $\Delta^i(\pi^J)$, then, since $i\leq \delta^a(\cC_{\ov \eta}) \leq \delta(\cC_{\ov \eta})$, Fact \ref{F:strati} implies that  $\cC_{\eta}$ is a nodal curve with $\delta(\cC_{\ov \eta})=i$;
hence we must also have that $\delta^a(\cC_{\ov \eta})= \delta(\cC_{\ov \eta})$ which implies that $\cC_{\ov \eta}$ is irreducible. 

Statement \eqref{dischilb}: the first part  follows from Theorem \ref{thm:smoothness_hilb} while the second part follows from Fact \ref{F:strati}. 
\end{proof}

\subsection{The sheaf ${\rm Irr}(X/Y)$}\label{sheafirr}

We shortly discuss the sheaf of irreducible components of a family of curves.
Let $f:X \to Y$ be a proper family of geometrically reduced curves.  
By  \cite[Prop. 6.2]{N0} applied to the restriction $f_{\rm sm}:X_{\rm sm} \to Y$ of $f$ to the smooth locus, 
there is a constructible sheaf 
${\rm Irr}(X/Y)$ of finite sets, whose stalk ${\rm Irr}(X/Y)_y$ at the point $y$ is the set of irreducible components of the fibre 
$X_y=f^{-1}(y)$. 
Let  $\{Y_\alpha \}$ be the stratification of $Y$ such that ${\rm Irr}(X/Y)_{|Y_\alpha}$ is locally constant.
Let us fix $o\in Y$. Up to shrinking $Y$ we may assume that every stratum contains $o$ in its closure.

Since the fibre $X_{o}=f^{-1}(o)$ is geometrically reduced, we may find, after shrinking 
$Y$ again, a set  $\{\sigma_v\}_{v\in {\rm Irr}(X_o)}: Y \to X$ of sections of the family such that:  

\begin{itemize}
\item 
the point 
$\sigma_v(o)$ belongs to the smooth locus of the irreducible component corresponding to $v$;
\item 
for every $v$ and for every $y\in Y$, the point $\sigma_v(y)$ lies in the smooth locus of $X_y$, hence it belongs to a unique irreducible component of $X_y$.
\end{itemize}
Therefore, we get a map of sets

$$
\begin{aligned}
{\mathcal V}_y:{\rm Irr}(X_o)={\rm Irr}(X/Y)_o & \longrightarrow {\rm Irr}(X/Y)_y ={\rm Irr}(X_y)  \\
v & \mapsto \text{ irreducible component of } X_y \text{ that contains } \sigma_v(y),
\end{aligned}$$
defined for $y$ in a neighborhood of $o$. By the hypothesis on the strata, {\em this map is surjective}.
It follows in particular that, on an appropriate neighborhood of every point, the restriction of the sheaf
${\rm Irr}(X/Y)$ to the connected components of the strata containing the point in their closure is not only locally constant but in fact constant.
More precisely, for every point $y \in Y$ there is a partition $\lambda_y$ of ${\rm Irr}(X_o)$
$${\rm Irr}(X_o)=\coprod_{a \in {\rm Irr}(X_y)} V_a$$  
defined by $V_a:={\mathcal V}_y^{-1}(a).$ Let $Y_\lambda \subseteq Y$ be the locally closed subset of points $y\in Y$ such that $\lambda_y=\lambda$. 
The choice of a section in every subset of the partition gives a trivialization of the restriction of  ${\rm Irr}(X/Y)$ to $Y_\lambda$. We summarize the discussion above in the following

\begin{proposition}\label{sheafirrprop}
Let $f:X \to Y$ be a proper family of geometrically reduced curves, and let ${\rm Irr}(X/Y)$ its sheaf of irreducible components. 
For a point $o \in Y$, let $\cP_o$ be the set of partitions of the set ${\rm Irr}(X/Y)_o$ of irreducible components giving rise to a decomposition of $X_o$ into {\em connected} subvarieties. 
Then, there exists a neighborhood $U$ of $o$ in the \'etale topology endowed with  
a stratification $\{U_\lambda\}$, indexed by $\cP_o$, with the property that the restriction 
of the sheaf ${\rm Irr}(X/Y)$ to every $\{ U_\lambda \}$ is a constant sheaf of sets.
\end{proposition}

\begin{remark}\label{irr_spec_conn}
The restriction on the set of partitions stems from the fact that the specialization of an irreducible component is connected.
\end{remark}

\subsection{The families associated with a miniversal deformation}
\label{Indbroken}

We apply the considerations of section \ref{sheafirr} to versal families of curves. 

Let  $\pi:(\C, \rC) \to (B, o)$ be a projective versal deformation of the (reduced) curve with planar singularities $\rC$ over a connected variety $B$ (see Fact \ref{F:versal}\eqref{F:versal1}).
Up to passing to an open subset of $B$ containing $b$, we can assume that $\pi:\C\to B$ is a versal family of curves with locally planar singularities (see Fact \ref{F:versal}\eqref{F:versal2}), which implies that $B$ is smooth and irreducible (see the discussion that follows Fact \ref{F:versal}). Moreover, up to passing to a further Zariski open subset of $b$, we can assume that the family satisfies  the  conclusions in Proposition \ref{sheafirrprop}.

Let $V:=V(\rC)$ denotes the set of irreducible components of $\rC$. 
By Fact \ref{F:strati}, for any $d$, we have that the  locus $B^{\delta \geq d}_\times $ parameterizing nodal curves is open and 
dense in the stratum $B^{\delta \geq d}$.
By the discussion in \S \ref{sheafirr}, every curve $\cC_s$ of the family 
determines a partition $\lambda_s = \{V_{\alpha}\}_{\alpha\in V(\cC_s)}$ of $V(\rC)$, hence a decomposition of $\rC$ into a union of connected subcurves. 

\begin{remark}\label{case}
The partition associated to the generic (smooth)  fiber gives the partition associated with the connected components of $\rC$;
at the other extreme, the map $\mathcal V_s$ is a bijection for any fiber $\cC_s$ belonging to the equigeneric stratum 
(by Lemma \ref{irred_comp} below), hence it gives rise to the identity partition. 
More generally, if $\cC_{s'}$ is a specialization of $\cC_s$, then the map $\mathcal V_s$ factors through $\mathcal V_{s'}$, which implies that
  $\lambda_{s'}$ is a refinement  of $\lambda_s$.
\end{remark}

We start by proving the following result which is instrumental for defining the families we need to consider:

\begin{lemma}\label{irred_comp}
With the same assumptions as before, consider the equigeneric stratum
 of maximal cogenus,  $\Delta:=B^{\delta=\delta(\rC)}$,
and let ${\cC}_\Delta \to \Delta$ be the restriction of the universal family $\pi:\C\to B$ to $\Delta$. 
Then on  $\Delta$ the following properties hold true
\begin{enumerate}
\item \label{sheaf1}
the sheaf of sets  ${\rm Irr}(\cC _\Delta / \Delta) $ of the irreducible components is constant;
\item \label{sheaf2}
the sheaf of sets of connected subcurves is constant along $\Delta$.
\end{enumerate}
\end{lemma}
\begin{proof}
Let us first prove \eqref{sheaf1}. Consider the normalization  $\wt{\Delta}\to \Delta$ and denote by $\C_{\wt{\Delta}} \to \wt{\Delta}$  the pull-back of the family $\C_{\Delta}\to \Delta$.  
According to Fact \ref{F:sim-res}, the normalization $\wt{\C}_{\wt{\Delta}}\to \C_{\wt{\Delta}}$ is a simultaneous normalization of the family $\C_{\wt{\Delta}} \to \wt{\Delta}$.  In particular, the sheaf of 
connected components of the family $\wt{\C}_{\wt{\Delta}}\to \wt{\Delta}$, which coincides with the pull-back to $\wt{\Delta}$ of
 the sheaf of irreducible components of the family $\C_{\wt{\Delta}} \to \wt{\Delta}$, is locally constant on $\wt{\Delta}$, hence 
constant, in force of Proposition \ref{sheafirrprop}, since the central point belongs to the equigeneric stratum.

Let us now prove part \eqref{sheaf2}. 
From  \eqref{sheaf1}, we have that if $\cC_\Delta = \bigcup_{i=1}^N \cC^{(i)}_\Delta$ is the decomposition into irreducible components, then the decomposition into irreducible components
of the geometric fiber $\cC_{\ov t}$ over any point $t\in \Delta(\rC)$  equals  $\bigcup_{i=1}^N \cC^{(i)}_{\ov t}$. 
For each $t$, we have, by Hironaka's formula \cite[Lemma 3.3.2]{GLS}
\begin{equation}\label{hiro}
\delta(\cC_{\ov t})=\sum_{i=1}^N \delta(\cC^{(i)}_{\ov t})+ \sum_{1\leq k<l \leq N}| \cC^{(k)}_{\ov t}\cap \cC^{(l)}_{\ov t}|.
\end{equation}
The delta invariant and the intersection numbers of the subcurves are upper semicontinuous functions in flat families.
As the sum (\ref{hiro}) is constant, we have that $\delta(\cC^{(i)}_{\ov t})$ and $| \cC^{(k)}_{\ov t}\cap \cC^{(l)}_{\ov t}|$
don't depend on $t$. Assume $\bigcup_{i=1}^s \rC^{(i)}$ is a connected subcurve of the central fibre such that, for some $t$, $\bigcup_{i=1}^s \cC^{(i)}_{\ov t}$ is disconnected, namely, up to a renumbering,
 we have 
 \[
 \cC'_{\ov t} \bigcap  \cC'' _{\ov t} =\emptyset, \mbox{ with } \cC'_{\ov t}:= \left( \bigcup_{i=1}^a \cC^{(i)}_{\ov t} \right), \mbox{ and }  \cC'' _{\ov t} :=\left (\bigcup_{i=a+1}^s \cC^{(i)}_{\ov t} \right).
 \]
Denoting $\rC'=\bigcup_{i=1}^a \rC^{(i)}$ and $\rC''=\bigcup_{i=a+1}^s \rC^{(i)}$,  by the argument above we have $|\rC' \cap \rC ''|=|\cC'_{\ov t} \cap \cC ''_{\ov t}|=0.$ Since $\rC'$ and $\rC ''$ have no  common component, their intersection number is strictly positive unless the curves are disjoint, which would contradict the connectedness of $\bigcup_{i=1}^s \rC^{(i)}$.
\end{proof}

\subsection{Main Theorems}\label{main_thm}

Let   $\rC$ be a projective curve with planar singularities, defined over $\CC$ or over $\ov{\FF_\pi}$ with big enough cardinality. 
As in \S \ref{Indbroken}, consider a versal deformation $\pi:(\cC,\rC)\to (B,b)$ for $\rC$, small enough to satisfy the conclusions of Proposition \ref{sheafirrprop}. 
The index $(\,\,)_\times$ applied to subsets of $B$ denotes the operation of intersecting with the nodal locus.  

Consider any point $b \in \Delta_\times$: by Lemma \ref{irred_comp}(\ref{sheaf1}), $V:=V(\rC)$ is identified with $V(\cC_b)$.
For any partition $\lambda$ of $V$, giving a decomposition  $$\rC = \bigcup \rC_i$$ of $\rC$ into connected subcurves, we also have
a decomposition 
$$\cC_b = \bigcup \cC_{b, i}$$ of $\cC_b$, whose subcurves are connected by Lemma \ref{irred_comp} (\ref{sheaf2}).

\begin{notation}\label{familyday}
For $b \in \Delta_\times$ and $\lambda$ a partition of the set $V$ decomposing $\rC$ in connected subcurves, we let:
\begin{enumerate}
\item 
 $E_\lambda$ be the set of nodes joining the different subcurves, i.e. $E_\lambda= \bigcup _{i \neq j} \cC_{b,i} \bigcap \cC_{b,j},$ and set $\delta(\lambda):=|E_\lambda|$.

\item
$B_\lambda \subseteq  B_\times$  be the set where the nodes in $E_{\lambda}$ persist.

\item
$\pi_\lambda: \cC_{\lambda} \to B_{\lambda}$ be the family of reduced nodal curves
obtained by normalizing these nodes. Notice that the subcurves are now disjoint. 

\item
$B_{\lambda, \reg} \subseteq B_\lambda$ be the open dense set where the family 
$$
\pi_\lambda:{\cC_{\lambda}}_{|{B_{\lambda, \reg}}} \to B_{\lambda, \reg},
$$ 
is smooth. It is the subset of 
$B_\times$ where precisely the nodes in $E_\lambda$ persist while the others are smoothed.

\item
$i_\lambda: B_{\lambda, \reg} \to B$ be the (locally closed) embedding.
\end{enumerate}
\end{notation}

\begin{remark}
It is clear that this construction does not depend on the choice of $b$. 
Furthermore, if a partition $\mu$ refines the partition $\lambda$, then $E_\lambda \subset E_\mu$, hence the locus  $B_\mu$ is naturally contained in  $B_\lambda$, whereas the curves in $\cC_\mu$  are clearly partial normalizations of those in $(\cC_\lambda)_{|B_\mu}$, as they are obtained from these latters by normalizing other nodes.
\end{remark}

\begin{theorem} \label{thm:hilb}
Let $\pi:(\cC,\rC)\to (B,b)$ be as above and refer to Notation \ref{familyday}. 
Let
$$
\pi_\lambda^{[n]}:{\cC_{\lambda}}^{[n]}_{|{B_{\lambda, \reg}}}\to B_{\lambda, \reg},
$$
the associated relative Hilbert scheme of length $n$ (which coincide with the $n$-th relative symmetric product since $\pi_\lambda$ is smooth over $B_{\lambda, \reg}$),
and set 
\[
\fF_{\lambda}^{[n]}:=\bigoplus_i \left(\left({\iota_\lambda}\right)_{!*}R^i{\pi_\lambda^{[n]}}_* \Qlbar \right)[-i].
\]
Then we have

\begin{equation}\label{vivek_formula}
R\pi_*^{[n]} \Qlbar \cong \bigoplus_{\lambda \in \cP}  \fF_{\lambda}^{[n-\delta(\lambda)]} [-2\delta(\lambda)](\delta(\lambda))
\end{equation}
where $\cP$ is the set of partitions of the set $V(\rC)$ decomposing $\rC$ in connected subcurves.
\end{theorem}
\begin{proof}
We descend to a family $\pi_o:\cC_o \to B_o$ defined over a finite, big enough field $\FF_{\pi}$. 
Since the sheaf of irreducible components is constant along the stratum $\Delta$ of maximal cogenus by Lemma \ref{irred_comp}\eqref{sheaf1}, we can also assume, up to passing to a bigger finite field, that the geometric  irreducible components of the closed fibers of $\pi_o$ are defined over the base field $\FF_{\pi}$.

By the classical MacDonald's formula (Equation \ref{macdonald}), for every $\lambda \in \cP$ we have:
\begin{equation}\label{E:plambda}
\sum_n q^nR{\pi_\lambda^{[n]}}_* \Qlbar  =  \frac{\Lambda^*(-q R^1 {\pi_\lambda}_* \Qlbar)}
{\Lambda^*\left(-q\left(R^0 {\pi_\lambda}_* \Qlbar(1 + \LL \right)\right)}.
\end{equation}
Since the local system $R^0 {\pi_\lambda}_* \Qlbar$ is constant on $B_{\lambda, \reg}$, 
the effect of the denominator results only in some shifts, direct sums and Tate twists, hence ininfluent to the computation of 
$({\iota_\lambda})_{!*}$. 
Using formula \eqref{E:plambda} and  applying Theorem \ref{vivek_nodal} together with Remark \ref{R:classCKS}, we deduce  that at every point $b \in {B_o}_\times(\FF_{\pi^r})$ 
the traces of the powers of the Frobenius map on the stalks of the two sides of \eqref{vivek_formula} coincide.  
Now, applying Corollary \ref{works} of \S \ref{appen} we have the isomorphism \eqref{vivek_formula} on the whole nodal set 
$B_\times$.
Since the nodal set is dense in every higher discriminant by Theorem \ref{hd_families}\eqref{dischilb},  the isomorphism \eqref{vivek_formula} holds on the entire $B$.
\end{proof}

\begin{example}\label{ex:vivekformula}
Let $\rC$ be the union of pair of lines, $\rC_1, \rC_2$ which meet once and transversely.  
A representative for the base $B$ of a versal deformation of $\rC$ is given by 
taking the compactification of the map $(x,y) \mapsto xy$; in any case we denote this
deformation by $(\cC, \rC) \to (B,o)$.  

We want to compute directly the LHS and RHS of Theorem \ref{thm:hilb}.We will just
study the stalks at the point $[\rC]$. One has, e.g. from
\cite{Ran}, 
$$[\rC^{[n]}]=[(\PP^1)^{[n]}] + [(\PP^1 \coprod \PP^1)^{[n-1]})] \cdot \LL .$$
Hence, passing to the generating series, the LHS is given by: 
\begin{eqnarray*}
\bigoplus_{n=0}^\infty q^n R\pi^{[n]}_{*} \overline{\QQ}_l |_{[\rC]} & = & 
\sum_{n = 0}^\infty q^n \left( [\PP^n] + \LL \cdot  \sum_{j=0}^{n-1} [\PP^j]  [\PP^{n-1-j}] \right)
\\
& = & \frac{1}{(1-q)(1-q\LL)} + \frac{q \LL}{\left((1-q)(1-q\LL)\right)^2}.
\end{eqnarray*}


On the RHS, we are reduced to summing over decompositions
of the curve $\rC$; here there are just two, $\rC = \rC$ and $\rC = \rC_1 \cup \rC_2$, with $\delta(\lambda)=1$.
For this latter decomposition the stratum $B_\lambda$ is just the point $o$. 
All genera are zero and (hence) all fine compactified Jacobians are just points.  Thus, 
the contribution of $\rC = \rC$ is $ \frac{1}{(1-q)(1-q\LL)}$
and 
the contribution of $\rC_1 \cup \rC_2$ is $\left(\frac{1}{(1-q)(1-q\LL)}\right)^2$, whith a term $q\LL$ to account for the shifts in \ref{thm:hilb} . 
\end{example}

\begin{theorem} \label{thm:jac_vf}
Let $\pi:(\cC,\rC)\to (B,b)$ be as above and let 
 $\pi^J:\ov{J}_{\cC}\to B$ be a relative fine compactified Jacobian (which exists after passing to an \'etale cover of $B$, by Theorem \ref{T:fam-Jac}).
Then, if $j: B_{\reg} \to B$, we have
\begin{equation}
R\pi^J_* \Qlbar =\bigoplus_i j_{!*}\left(\textstyle\bigwedge^i R^1\pi_*{\Qlbar}_{|{\Breg}}\right)[-i]
\end{equation}  
i.e. no summand of $R\pi^J_* \Qlbar$ has positive codimensional support.
\end{theorem}
\begin{proof}
Over $\Breg$ the isomorphism $R\pi^J_* {{\Qlbar}_{|\Breg}} =\bigoplus_i \textstyle \bigwedge^i R^1\pi_*{{\Qlbar}_{|\Breg}}[-i]$
follows from the standard computation of the cohomology of the Jacobian of a nonsingular curve.
Hence $R\pi^J_* \Qlbar$ contains a summand isomorphic to $\bigoplus_i j_{!*}\left(\textstyle\bigwedge^i R^1\pi_*{\Qlbar}_{|\Breg}\right)[-i]$.
Assume by contradiction that there are other summands in the decomposition theorem: these must be supported on some 
codimension $i>0$ irreducible component of $\Delta^i(\pi^J)$ by Theorem \ref{hd_general}.  Theorem \ref{hd_families}\eqref{discrcompjac} implies that the generic point $\eta$ of this support is 
such that $\C_{\ov \eta}$ is an {\em irreducible} nodal curve.
Since the stalk at $\eta$ of the new summand is a complex of pure vector spaces, this would imply that the weight polynomial of the compactified Jacobian of $\C_{\eta}$ and
$ \mathfrak{w}\left( \sum_i IC\left(\textstyle\bigwedge^i R^1\pi_*{\Qlbar}_{|\Breg}\right)_\eta[-i]\right)$
disagree. But both polynomials are  equal to $(1 + t)^{2 g({\cC_{\ov \eta}^\nu})} t^{2  h^1(\Gamma)}$, where $\cC_{\ov \eta}^{\nu}$ is the normalization of the curve $\cC_{\ov \eta}$ (see Corollaries \ref{cor:weightpolyic} and \ref{C:weightJ}), and this is the desired absurd.

\end{proof} 

\begin{remark}\label{alternative}
In the appendix \ref{appen2} we will compute the weight polynomial of a fine compactified Jacobian of a general nodal curve, i.e. not necessarily irreducible.
The comparison with (\ref{weight_cks}) gives an alternative  proof of Theorem \ref{thm:jac_vf} which avoids the estimate on the dimension of the higher discriminants
of Theorem \ref{hd_families}\eqref{discrcompjac}. The proof given here, though, seems more conceptual to us, as it emphasizes the link between supports theorems and deformation theory.
\end{remark}

\subsection{Independently broken H-smooth families}
In this section we consider a class of not necessarily versal families of curves.
\begin{definition}\label{ind_brok_H_smooth}
A projective family $\pi: \cC \to B$ of curves with planar singularities is said independently broken H-smooth
if \begin{enumerate}
\item 
All the relative Hilbert schemes $\pi^{[n]}: \cC^{[n]} \to B$ have nonsingular total space (included the case $n=0$, i.e. $B$ is nonsingular),
and there exists a relative fine compactified Jacobian.
\item
The sheaf of irreducible components ${\rm Irr}(\cC/B)$ satisfies the conclusions of Proposition \ref{sheafirrprop}, 
\item
For every $d$, the set  $B^{\delta=d}:= \{b \in B  \: : \:  \delta(\cC_b)=d \}$ contains an open dense subset $B^{\delta=d}_\times$ parameterizing 
nodal curves.
\end{enumerate}
\end{definition}

\begin{example}
Let  $\rC$ be a projective curve with planar singularities and let $\pi:(\cC, \rC)\to (B,b)$ be a projective versal deformation of $\rC$ over a variety $B$. 
Pick a subspace $\UU \subset B$ 
of dimension at least $\delta(\rC)+1$ transverse to $\Delta$. If $\UU$ is small enough, the restriction of the versal family to $\UU$ 
gives an independently broken H-smooth family by Theorem \ref{thm:smoothness_hilb}.
Viceversa, an independently broken H-smooth family is locally the pullback along a smooth morphism of such a family.
\end{example}

Remark that, in view of Corollary \ref{hil_comp}, the total space of any relative fine compactified Jacobian for an independently broken H-smooth family 
 is nonsingular.
It is almost immediate to notice that the two main theorems in \S \ref{supp_thm_vf} hold for independently broken H-smooth families.
First notice that the constructions leading to the definitions of the loci $B_\lambda$, the families $\pi_\lambda$ may still be done.
Noticing that the higher discriminants are just the intersections of those for the versal family we easily see:

\begin{corollary}
Theorems \ref{thm:jac_vf} and \ref{thm:hilb} hold for  hold for independently broken H-smooth families.
\end{corollary}

\section{Appendix 1}\label{appen}
We collect here some consequences of the results contained in \S 5.3 of \cite{BBD} to justify our reduction to point counting.

In this appendix $B_o$ denotes an algebraic variety defined over the finite field $k=\FF_\pi$ and we will be considering perverse $\Qlbar$-sheaves (or more generally complexes of constructible $\Qlbar$-sheaves) on $B_o$ that are \emph{pure} in the sense of \cite[\S 5.1]{BBD}.
However, recall that we use (as always throughout this paper) a different convention on perverse sheaves with respect to \cite{BBD}: a sheaf $K$ supported on an irreducible closed subvariety $Y_o\subseteq B_o$ is perverse in our convention if and only if $K[\dim Y_o]$ is perverse in the sense of \cite{BBD}.

We will need the following two results  from \cite[\S 5.3]{BBD} on the structure of pure perverse sheaves on $B_o$.

\begin{proposition}\label{fact1}(\cite[Thm. 4.3.1, Prop. 5.3.9]{BBD})
A pure perverse sheaf $P_o$ on $B_o$  admits a unique decomposition 
\begin{equation*}\label{E:dec-pure}
P_o= \bigoplus_i S_i\otimes E_{n_i},
\end{equation*}
where $S_i$ are simple pure perverse sheaves on $B_o$ and $E_k$ is the rank $k$ Jordan block locally constant $\Qlbar$-sheaf described in \cite[p. 138]{BBD}.

Moreover, each $S_i$ is of the form $j_{!*}(L_i)$, where $j:U_{o, i}\hookrightarrow B_o$ is a locally closed embedding, $U_{o, i}$ is smooth and irreducible, and $L_i$ is a $\Qlbar$-sheaf lisse 
and irreducible on $U_{o,i}$. In particular, the support of $S_i$ is the irreducible closed subvariety $\ov U_{o, i}$. 
\end{proposition}

The supports of the simple pure perverse sheaves appearing in the decomposition \eqref{E:dec-pure} of $P_o$ are called the \emph{supports} of $P_o$ (note that the supports are irreducible closed subvarieties of $B_o$). The \emph{semisemplification} of $P_o$ is given in terms of the decomposition \eqref{E:dec-pure} as
\begin{equation*}
P_o^{ss}=\bigoplus_i S_i^{n_i}.
\end{equation*}

\begin{proposition}\label{fact2}(\cite[Cor. 5.3.11]{BBD})
If $P_o$ is a pure perverse sheaf, and $j:U_o \to B_o$ is a dense open imbedding, then
\begin{equation*}
P_o=j_{!*}j^*P_o \oplus P'
\end{equation*}
where $P'$ is a perverse pure sheaf supported on $B_o \setminus U_o$.
\end{proposition}

Using the above results, we can give the following criterion ensuring that two perverse pure sheaves have isomorphic semisemplifications.

\begin{proposition}
Let $P_o$ and $Q_o$ two pure perverse sheaves on $B_o$, and let $\{Y_{o,\alpha} \}_{\alpha=1}^l$ be the collection of the supports of $P_o$ and $Q_o$.
Assume that, for every $\alpha=1, \ldots, l$, there is a dense open subset 
$U_{o, \alpha} \subseteq Y_{o,\alpha}$, with the following property:
for every $x \in U_{o, \alpha}(k')$ with $k'$ a finite extension of $k$, and for every positive integer $N$, we have
$$
Tr(\sigma_x^N, P_x)=Tr(\sigma_x^N, Q_x)
$$
where  $\sigma_x$ is the Frobenius conjugacy class in   $\pi_1(U_{o, \alpha})$ associated to $x$.
Then $P_o$ and $Q_o$ have isomorphic semisimplifications.

In particular,  the two sheaves $P_o$ and $Q_o$ have the same traces of the Frobenius everywhere, i.e.
$$Tr(\sigma_x^N, P_x)=Tr(\sigma_x^N, Q_x),$$
for  every  point $x \in B_o(k')$ with $k'$ any finite extension of $k$, and for every positive integer $N$.
\end{proposition}
\begin{proof}
The proof is by induction on the number of supports. Consider a maximal support (i.e. a support that is not contained in any other support),  say $Y_{o, 1}$ up to renaming the supports. Consider an open dense subset $j:U_{o,1}\hookrightarrow Y_{o,1}$ as in the hypothesis. By the maximality of $Y_{o,1}$ and the fact that $Y_{o,1}$ is irreducible, we can assume, up to passing to a smaller open subset, that $U_{o,1}$ is smooth and disjoint from all the supports different from $Y_{o,1}$. 
Combining Propositions \ref{fact1} and \ref{fact2}, we can write (up to further restricting $U_{o,1}$):
\begin{equation}\label{doubledec}
\begin{sis}
& P_o =  j_{!*}(j^*(P_o)) \oplus P_o' \, \text{ with  } \, j^*(P_o) =  \bigoplus_i L_i\otimes  E_{n_{i}},\\
& Q_o=  j_{!*}(j^*(Q_o))\oplus Q_o' \, \text{ with  } \, j^*(Q_o) =  \bigoplus_i M_i\otimes  E_{m_{i}},
\end{sis}
\end{equation}
where $L_i$ and $M_i$ are $\Qlbar$-sheaf lisse  and irreducible on $U_{o,1}$,  $n_i$ and $m_i$ are natural numbers, $P_o'$ and $Q_o'$ are pure perverse sheaves supported on $B_o\setminus U_{o,1}$.

The $\Qlbar$-sheaves $j^*(P_o)$ and  $ j^*(Q_o)$ are lisse on $U_{o,1}$ and they have the same traces of Frobenius everywhere on $U_{o, 1}$ by the hypothesis and the fact that $U_{o,1}$ is disjoint from all the supports different from $Y_{o,1}$. 
 Hence we can  apply Chebotarev theorem (see \cite[Thm. 1.1.2, Prop 1.1.2.1]{L-tf}) in order to conclude that $j^*(P_o)$ and $j^*(Q_o)$ have the same semisemplification, i.e. 
  \begin{equation}\label{E:samess}
\bigoplus_i L_i^{n_i} =j^*(P_o)^{ss}= j^*(Q_o)^{ss} =\bigoplus_i M_i^{m_i}. 
\end{equation}
In particular, $j^*(P_o)$ and $j^*(Q_o)$ have the same traces of Frobenius everywhere on $\ov{U_{i,o}}=Y_{o,i}$. 
This implies that the two pure perverse sheaves $P_o'$ and $Q_o'$ verify the same hypothesis on the traces of  Frobenius with respect to their supports $\{Y_{o,\alpha}\}_{\alpha=2}^l$. Hence by the induction hypothesis on the number of supports, we have that 
\begin{equation}\label{ssinduc}
(P_o')^{ss}=(Q_o')^{ss}.
\end{equation}
Combining \eqref{doubledec}, \eqref{E:samess} and \eqref{ssinduc}, we conclude that $P_o^{ss}=Q_o^{ss}$.

\end{proof}

\begin{corollary}\label{works}
Let $K_o$ and $L_o$ two pure complexes of constructible $\Qlbar$-sheaves on $B_o$ such that
$$
K_o \simeq \bigoplus \, ^p\cH^i(K_o)[-i],\, \, L_o \simeq \bigoplus \, ^p\cH^i(L_o)[-i].
$$
Let $\{Y_{o,\alpha} \}_{\alpha=1, \cdots l}$ be the collection of the supports of $^p\cH^i(K_o)$ and $^p\cH^i(L_o)$.
Assume that, for every $\alpha=1, \cdots l$, there is a dense open subset 
$U_{o, \alpha} \subseteq Y_{o,\alpha}$, with the following property:
for every $x \in U_{o, \alpha}(k')$ with $k'$ a finite extension of $k$, and for every positive integer $N$, we have
$$
Tr(\sigma_x^N, K_x)=Tr(\sigma_x^N, L_x)
$$
where  $\sigma_x$ is the Frobenius conjugacy class in  
$\pi_1(U_{o, \alpha})$ associated to $x$.
Then $K_o$ and $L_o$ have isomorphic semisimplifications.
\end{corollary}

\begin{proof}
One proceed by induction, starting with the open set on which $K_o$ and $L_o$ are isomorphic 
to a direct sum of pure semisimple (shifted) lisse sheaves. 
Then, using the fact that every summand is pointwise pure on an open set of its support, 
one can separate the different perversities according to the absolute values of the Frobenius eigenvalues.

\end{proof}

\section{Appendix 2}\label{appen2}

In this appendix, we work over an algebraically closed field 
$k = \overline{k}$.  Our goal is to determine
the class of a fine compactified Jacobian of a nodal curve $\rC$ in
$K_0(Var_{\overline{k}})$. As explained in Remark \ref{alternative} this computation gives an alternative proof of 
\ref{thm:jac}, and in turns it is a consequence of it. We include it for completeness, as we believe is of independent interest.

Let us first compute the class in $K_0(Var_{\overline{k}})$ of the generalized Jacobian $J_C$ of $C$, which is by definition the connected component of the Picard scheme $\Pic(C)$ of $C$ containing the identity. The normalization morphism $\nu: \rC^{\nu}\to \rC$ induces the sequence
$$1 \to \Gm \to \nu_* \Gm \to \nu_* \Gm/ \Gm \to 1,$$
which yields by taking cohomology:
\begin{equation}\label{E:seqGm}
1 \to H^0(\rC, \Gm) \to H^0({\rC^\nu}, \Gm) \to
H^0(\rC, \nu_* \Gm/ \Gm)
\to H^1(\rC, \Gm) \to H^1({\rC^\nu}, \Gm)  \to 1.
\end{equation}
In terms of the dual graph $\Gamma=\Gamma_{\rC}$ of $\rC$, we have
$$1 \to H^0(\Gamma, \ZZ)\otimes \Gm \to H^0({\rC^\nu}, \Gm) \to
H^0(\rC, \nu_* \Gm/ \Gm) \to H^1(\Gamma, \ZZ)\otimes \Gm \to 1.$$

Substituting into \eqref{E:seqGm} and restricting to the
connected component of the identity gives an exact sequence of algebraic groups
\begin{equation}\label{E:seqJ}
1\to H^1(\Gamma, \ZZ)\otimes \Gm\cong \Gm^{h^1(\Gamma)} \to J_{\rC} \stackrel{\nu^*}{\to} J_{\rC^{\nu}}\to 1,
\end{equation}
where $h^1(\Gamma)$ is the rank of the free abelian group $H^1(\Gamma, \ZZ)$.

Since $\mathbb{G}_m=\mathrm{GL}_1$ is a special group, the sequence \eqref{E:seqJ} is Zariski locally trivial, hence we have the following equality in $K_0(Var_{\overline{k}})$:
\begin{equation}\label{E:classJ}
J_{\rC} = J_{{\rC^\nu}} \mathbb{G}_m^{h^1(\Gamma)}=J_{{\rC^\nu}} (\LL-1)^{h^1(\Gamma)}.
\end{equation}

In order to compute the class in $K_0(Var_{\overline{k}})$ of a fine compactified Jacobian
$\ov J_{\rC}(\pol)$ of $\rC$, we need to recall the stratification of $\ov J_{\rC}(\pol)$ in terms
of partial normalizations of $\rC$ studied in \cite{MV} (see also \cite{OS, A}).
Given any torsion free, rank-$1$ sheaf $\I$ on $C$, its endomorphism sheaf $\underline{\mathrm{End}}_{\oO_{\rC}}(\I)$ is a sheaf of finite $\oO_{\rC}$-algebras such that
$\oO_{\rC} \subseteq \underline{\mathrm{End}}_{\oO_{\rC}}(\I) \subseteq \oO_{{\rC^\nu}}$. The sheaf $\I$ is naturally a sheaf on the partial normalization $ \rC^{\I}:= \underline{\mathrm{Spec}}_{\rC}( \mathrm{End}_{\oO_{\rC}}(\I) )$ of $C$; the original $\I$ being
recovered by the pushforward along the partial normalization morphism $\nu_{\I}: \rC^{\I} \to \rC$.
Since $\rC$ is nodal, it can be checked that $\rC^{\I}$ is the partial normalization of $\rC$ at all the nodes where $\I$ is not locally free and $\I$ is a line bundle on $\rC^{\I}$.
This gives rise to a stratification of any fine compactified Jacobian $\ov J_{\rC}(\pol)$
into locally closed subsets
\begin{equation}\label{E:stratJ}
\ov J_{\rC}(\pol)=\coprod_{S\subset \rC_{\rm sing}} \ov J_{\rC, S}(\pol):=
\coprod_{S\subseteq \rC_{\rm sing}} \{\I\in \ov J_{\rC}(\pol)\: : \: C^\I=C^S\}.
\end{equation}

The following result describes  the stratum $\ov J_{\rC, S}(\pol)$ in terms of the graph $\Gamma\setminus S$ obtained from the  dual graph $\Gamma=\Gamma_{\rC}$ of $\rC$ by deleting the edges corresponding to $S$.

\begin{proposition}\label{F:orb} (\cite[Thm. 5.1]{MV}) Let $\rC$ be a connected nodal curve over $\ov k$ and let $\overline{J}_{\rC}(\pol)$ be a fine compactified Jacobian. Then for every $S\subseteq \rC_{\rm sing}$,
the stratum $\ov J_{\rC, S}(\pol)$ is isomorphic to a disjoint union of
$\hat{c}(\Gamma \setminus S)$ copies of $J_{\rC^S}$, where
\begin{equation}\label{E:cG}
\hat{c}(\Gamma\setminus S)=
\begin{cases}
c(\Gamma\setminus S)=\# \{\text{spanning trees of } \Gamma\setminus S\} & \: \text{ if } \Gamma\setminus S  \text{ is connected,}\\
0 & \: \text{ if } \Gamma\setminus S  \text{ is not connected.}\\
\end{cases}
\end{equation}
\end{proposition}

We are now ready to compute the class of a fine compactified Jacobian of a nodal curve in $K_0(Var_{\overline{k}})$.

\begin{proposition} \label{prop:motiveofjbar}
Let $\rC$ be a connected nodal curve over $\overline{k}$ and let $\ov J_{\rC}$ be a fine compactified Jacobian of $C$.  Then, in $K_0(Var_{\overline{k}})$ ,
we have
\begin{equation}
\overline{J}_{\rC}(\pol)= J_{{\rC^\nu}} \cdot c(\Gamma) \LL^{h^1(\Gamma)}.
\end{equation}

\end{proposition}
\begin{proof}
From the stratification \eqref{E:stratJ} together with Proposition \ref{F:orb} and \eqref{E:classJ}, we get that
$$\overline{J}_{\rC}(\pol) = \sum_{S \subset \rE} \hat{c}(\Gamma \setminus S) \cdot J_{\rC^S} = J_{{\rC^\nu}}
\sum_{S \subset \rE} \hat{c}(\Gamma \setminus S) \cdot (\LL - 1)^{h^1(\Gamma \setminus S)}.$$
Thus our goal is to prove
$$\hat{c}(\Gamma) \LL^{h^1(\Gamma)} = \sum_{S \subset \rE} \hat{c}(\Gamma \setminus S) \cdot (\LL - 1)^{h^1(\Gamma \setminus S)}.$$
Note that if $\hat{c}(\Gamma \setminus S)$ is not zero, i.e. if $\Gamma\setminus S$ is connected, then $h^1(\Gamma \setminus S) = h^1(\Gamma) - |S|$.
We substitute $x + 1 = \LL$.  Then the above required formula reads
$$ \hat{c}(\Gamma) \sum_{i=0}^{h^1(\Gamma)} {h^1(\Gamma) \choose i} x^i  = \sum_{S \subset \rE} \hat{c}(\Gamma \setminus S) \cdot x^{h^1(\Gamma) - |S|}.$$
This holds for each coefficient of $x$ by the following Lemma \ref{lem:subsum}.
\end{proof}

\begin{lemma} \label{lem:subsum}
For any connected graph $\Gamma$,
\begin{equation*}
\sum_{\substack{S\subseteq \rE(\Gamma) \\ |S|=i}} \hat{c}(\Gamma\setminus S)=\binom{b_1(\Gamma)}{i}\cdot \hat{c}(\Gamma). 
\end{equation*}
\end{lemma}
\begin{proof}
The LHS counts the number of ways to first remove $i$ edges from $\Gamma$, and then find a spanning tree of $\Gamma$ from what remains,
whereas the RHS counts the number of ways to first find a spanning tree of $\Gamma$, which amounts to removing some $b_1(\Gamma)$ edges, and
then decide which $i$ of those edges you removed `first'.
\end{proof}

From the above Proposition, we can compute the weight polynomial of fine compactified Jacobians of nodal curves.

\begin{corollary}\label{C:weightJ}
Same assumptions as in Proposition \ref{prop:motiveofjbar}. Then the weight polynomial of $\ov J_{\rC}(\pol)$ is equal to
\begin{equation}\label{weightcompjac}
 \mathfrak{w}\left(\ov J_{\rC}(\pol)\right) =
(1 + t)^{2 g({\rC^\nu})} t^{2  h^1(\Gamma)} c(\Gamma).
\end{equation}

\end{corollary}
\begin{proof}
This follows from Proposition \ref{prop:motiveofjbar} using that $\mathfrak{w}(\LL)=t^2$ and that $\mathfrak{w}(J_{\rC^{\nu}})=(1+t)^{2g^{\nu}(\rC)}$ because $\mathfrak{w}(J_{\rC^{\nu}})$ is an abelian variety of
dimension $g^{\nu}(\rC)$.

\end{proof}

\end{document}